\documentclass[a4paper,12pt, abstract=true]{scrartcl}

\usepackage[ansinew]{inputenc}
\usepackage[english]{babel}
\usepackage{latexsym,amssymb,amsfonts,amsmath,bbm}
\usepackage{theorem}
\usepackage{xcolor}
\usepackage{graphicx}
\usepackage{esint}
\usepackage{stmaryrd}
\usepackage[normalem]{ulem} 

\numberwithin{equation}{section}
\theoremstyle{plain}

\newtheorem{dummytheorem}{Dummy-Theorem}[section]
\newcommand{\proofendsign}{$\Box$} 

\newtheorem{lemma}[dummytheorem]{Lemma}
\newtheorem{theorem}[dummytheorem]{Theorem}
\newtheorem{proposition}[dummytheorem]{Proposition}

\newtheorem{corollary}[dummytheorem]{Corollary}

\newtheorem{example}[dummytheorem]{Example}

\newtheorem{remark}[dummytheorem]{Remark}

\newenvironment{proof}{{\noindent \textbf{Proof} }}
 {{\hspace*{\fill}\proofendsign\par\bigskip}}

\newcommand{\Flinks}{F^{\leftarrow}}

\newcommand{\hF}{\widehat{F}}

\newcommand{\NNN}{\mathbb{N}}

\newcommand{\R}{\mathbb{R}}
\newcommand{\RRR}{\mathbb{R}}

\newcommand{\F}{\mathbb{F}}
\newcommand{\FFF}{\mathbb{F}}

\newcommand{\TTT}{\mathbb{T}}
\newcommand{\EEE}{\mathbb{E}}

\newcommand{\pr}{\mathbb{P}}
\newcommand{\MP}{\mathbb{P}}
\newcommand{\MQ}{\mathbb{Q}}

\newcommand{\ex}{\mathbb{E}}
\newcommand{\vari}{\mathbb{V}\textrm{ar}}
\newcommand{\covi}{\mathbb{C}\textrm{ov}}

\newcommand{\eins}{\mathbbm{1}}

\newcommand{\cB}{\mathcal B}

\newcommand{\cF}{\mathcal F}

\newcommand{\cH}{\mathcal H}

\newcommand{\cK}{\mathcal K}

\newcommand{\cM}{\mathcal M}
\newcommand{\cP}{\mathcal P}

\newcommand{\cR}{\mathcal R}

\newcommand{\cU}{\mathcal U}

\newcommand{\OFP}{(\Omega,{\cal F},\pr)}
\newcommand{\tOFP}{(\overline{\Omega},\overline{{\cal F}},\overline{\pr})}

\def\bcswitch{\left\{\renewcommand{\arraystretch}{1.2}\begin{array}{c@{,~}c}}

\def\ecswitch{\end{array}\right.}

\newcommand{\argmin}{\operatornamewithlimits{argmin}}
\begin{document}
\title{First order asymptotics of the sample average approximation method to solve risk averse stochastic programs}
\author{
Volker Krätschmer\footnote{Faculty of Mathematics, University of Duisburg--Essen, { volker.kraetschmer@uni-due.de}}
}
\date{}

\maketitle
\begin{abstract}
We investigate statistical properties of the optimal value of the Sample Average Approximation of stochastic programs, continuing the study \cite{Kraetschmer2022}. Central Limit Theorem type results are derived for the optimal value.  As a crucial point the investigations are based on a new type of conditions from the theory of empirical processes which do not rely on pathwise analytical properties of the goal functions. In particular, continuity or convexity in the parameter is not imposed in advance as usual in the literature on the Sample Average Approximation method. It is also shown that the new condition is satisfied if the paths of the goal functions are H\"older continuous so that the main results carry over in this case. Moreover, the main results are applied to goal functions whose paths are piecewise H\"older continuous as e.g. in two stage mixed-integer programs. The main results are shown for classical risk neutral stochastic programs, but we also demonstrate how to apply them to the Sample Average Approximation of risk averse stochastic programs. In this respect we consider stochastic programs expressed in terms of absolute semideviations and divergence risk measures.
\end{abstract}

\textbf{Keywords:} Risk averse stochastic program, Sample Average Approximation, absolute semideviations, divergence risk measures, covering numbers.
\section{Introduction}
Consider a classical risk neutral stochastic program 
\begin{equation}
\label{optimization risk neutral}
\inf_{\theta\in\Theta}\EEE\big[G(\theta,Z)\big],
\end{equation}
where $\Theta$ denotes a compact subset of $\RRR^{m}$, whereas $Z$ stands for a $d$-dimensional random vector with distribution $\MP^{Z}$. In general the parameterized distribution of the goal function $G$ is unknown, but some information is available by i.i.d. samples. Using this information, a general device to solve approximately  problem (\ref{optimization risk neutral}) is provided by the so-called \textit{Sample Average Approximation} (\textit{SAA}) (see \cite{ShapiroEtAl}). For explanation, let us consider a sequence $(Z_{j})_{j\in\NNN}$ of independent $d$-dimensional random vectors on some fixed complete atomless probability space $\OFP$ which are identically distributed as the $d$-dimensional random vector $Z$.
Let us set
$$
\hat{F}_{n,\theta}(t) := \frac{1}{n}~\sum_{j=1}^{n}\eins_{]-\infty,t]}\big(G(\theta,Z_{j})\big)
$$
to define the empirical distribution function $\hat{F}_{n,\theta}$ of $G(\theta,Z)$ based on the i.i.d. sample 
$(Z_{1},\cdots,Z_{n})$. Then the SAA method approximates the genuine optimization problem (\ref{optimization risk neutral}) by the following one
\begin{equation}
\label{SAAriskneutral}
\inf_{\theta\in\Theta}\int_{\RRR}t ~d\hat{F}_{n,\theta}(t) = \inf_{\theta\in\Theta}\frac{1}{n}~\sum_{j=1}^{n}G(\theta,Z_{j})\quad(n\in\NNN).
\end{equation}
The optimal values depend on the sample size and the realization of the samples of $Z$. Their asymptotic behaviour with increasing sample size, also known as the first order asymptotics of (\ref{optimization risk neutral}), is well-known. More precisely, 
the sequence of optimal values of the approximated optimization problem converges 
$\MP$-a.s. to the optimal value of the genuine stochastic program. Moreover, if  $G$ is Lipschitz continuous in $\theta$, and if \eqref{optimization risk neutral} has a unique solution, then the stochastic sequence 
\begin{equation}
\label{sequence}
\left(\sqrt{n}\Big[\inf_{\theta\in\Theta}\int_{\RRR}t ~d\hat{F}_{n,\theta}(t) - \inf_{\theta\in\Theta}\EEE\big[G(\theta,Z)\big]\Big]\right)_{n\in\NNN}
\end{equation}
is asymptotically normally distributed. 
In \cite{EichhornRoemisch2007} asymptotic distributions of this stochastic sequence have also be found for stochastic mixed-integer programs, where typically the objectives are not continuous in the parameter. For these results, and more on asymptotics of the SAA method the reader may consult the monograph \cite{ShapiroEtAl}, and in addition the contributions \cite{Pflug1999} and \cite{Roemisch2003}.
\medskip

In several fields like finance, insurance or microeconomics, the assumption of risk neutral decision makers are considered to be too idealistic. Instead there it is preferred to study the behaviour of actors with a more cautious attitude, known as risk aversion. In this view the optimization problem (\ref{optimization risk neutral}) should be replaced with a \textit{risk averse} stochastic program, i.e. an optimization problem
\begin{equation}
\label{optimization risk averse}
\inf_{\theta\in\Theta}\rho\big(G(\theta,Z)\big),
\end{equation}
where $\rho$ stands for a functional which is nondecreasing w.r.t. the increasing convex order. A general class of functionals fulfilling this requirement is built by the so called 
\textit{distribution-invariant convex risk measures} (see e.g. \cite{FoellmerSchied2011}, \cite{ShapiroEtAl}). They play an important role as building blocks in quantitative risk management (see \cite{McNeilEmbrechtsFrey2005}, \cite{PflugRoemisch2007}, \cite{Rueschendorf2013}), and they have been suggested as a systematic approach for calculations of insurance premia
(cf. \cite{Kaasetal2008}). \textit{Distribution-invariance} denotes the property that a functional $\rho$ has the same outcome for random variables with identical distribution. Hence, a distribution-invariant convex risk measure $\rho$ may be associated with a functional $\cR_{\rho}$ on sets of distribution functions (see e.g. \cite[Section 4.2]{KraetschmerSchiedZaehle2017}, and also \cite[(2.4)]{ClausKraetschmerSchultz2017}). In this case (\ref{optimization risk averse}) reads as follows
$$
\inf_{\theta\in\Theta}\cR_{\rho}(F_{\theta}),
$$
where $F_{\theta}$ is the distribution function of $G(\theta,Z)$. Then we may modify the SAA method by 
\begin{equation}
\label{SAA risk averse}
\inf_{\theta\in\Theta}\cR_{\rho}(\hat{F}_{n,\theta})\quad(n\in\NNN).
\end{equation}
As the title says, the subject of the paper is the first order asymptotics of the SAA method for (\ref{optimization risk averse}) with $\rho$ being a distribution-invariant convex risk measure.
It is already known that under rather general conditions on the mapping $G$ we have 
$$
\inf_{\theta\in\Theta}\cR_{\rho}\big(\hat{F}_{n,\theta}\big)\to \inf_{\theta\in\Theta}\cR_{\rho}\big(F_{\theta}\big)\quad\MP-\mbox{a.s.}
$$
(see \cite{Shapiro2013a}). In \cite{Kraetschmer2022} nonasymptotic upper estimates of
\begin{equation*}
\MP\Big(\Big\{\big|\inf_{\theta\in\Theta}\cR_{\rho}\big(\hat{F}_{n,\theta}\big) - \inf_{\theta\in\Theta}\cR_{\rho}\big(F_{\theta}\big)\big|\geq\varepsilon\Big\}\Big)\quad(n\in\NNN,\varepsilon > 0)
\end{equation*}
are derived, dependent on the sample size $n$. Besides the risk neutral case risk averse stochastic programs in terms of \textit{upper semideviations} and \textit{divergence risk measures} were considered there. As a by product uniform tightness of the 
stochastic sequence 
$$
\Big(\sqrt{n}\big[\inf_{\theta\in\Theta}\cR_{\rho}\big(\hat{F}_{n,\theta}\big) - \inf_{\theta\in\Theta}\cR_{\rho}\big(F_{\theta}\big)\big]\Big)_{n\in\NNN}
$$
is obtained in all cases. In this paper we continue this work, focussing on the asymptotic distributions of the stochastic sequence. 
To the best of our knowledge, this issue has been studied in \cite{GuiguesKraetschmerShapiro2018} and \cite{DentchevaEtAl2017} only. In both contributions, $G$ is assumed to be Lipschitz continuous in $\theta$, and a subclass of distribution-invariant convex risk measures of a specific form is considered. We shall extend the investigations in respect of allowing for more general goal functions $G$, which make possible to apply the results to stochastic programs whose goal functions are not continuous in the parameter. The value functions of two stage mixed-integer stochastic programs are prominent examples for such a type goal functions. Concerning the choice of distribution-invariant convex risk measures we shall restrict ourselves to \textit{absolute semideviations} and \textit{divergence risk measures}. These classes have only a small intersection with the class of distribution-invariant convex risk measures in \cite{GuiguesKraetschmerShapiro2018} but no one with the class in \cite{DentchevaEtAl2017}.
\medskip

The paper is organized as follows. We shall start with a general central limit theorem type result for the optimal values of classical risk neutral stochastic programs. The point is that we may extend this result if the SAA method is applied to risk averse stochastic programs. In Section \ref{FOA absolute semideviations} this will be demonstrated in the case that stochastic programs are expressed in terms of absolute semideviations, whereas in Section \ref{generalized divergence risk measures} the application to stochastic programs under divergence risk measures are considered. Our main asymptotic results are based on a technical result concerning the convergence of the sequence $\big(\sum_{j=1}^{n}G(\cdot,Z_{j})/n\big)_{n\in\NNN}$ in the path space. It will be formulated in Section \ref{basic result}. Finally Section \ref{Beweise} gathers proofs of results from the previous sections.
\par
The essential new ingredient of our results is to replace analytic conditions on the paths $G(\cdot,z)$ with requirements which intuitively make the family $\{G(\theta,Z)\mid\theta\in\Theta\}$ of random variables small in some certain sense. Fortunately, the respective invoked conditions are satisfied if the paths $G(\cdot,z)$ are H\"older continuous. We shall also see that we may utilize our results to study the SAA method for stochastic programs, where the paths $G(\cdot,z)$ are piecewise H\"older continuous but not necessarily continuous or convex. Value functions of two stage mixed-integer programs are typical examples for goal functions of such a kind. 
\section{First order asymptotics in the risk neutral case}
\label{Asymptotics risk neutral}
In this section we study the SAA \eqref{SAAriskneutral} associated with the risk neutral stochastic program \eqref{optimization risk neutral}. 
We shall restrict ourselves to mappings $G$ which satisfy the following properties.
\begin{itemize}
\item[(A1)] The set $\Theta$ is a compact subset of $\RRR^{m}$. The mapping $G$ is measurable w.r.t. the product $\cB(\Theta)\otimes\cB(\RRR^{d})$ of the Borel $\sigma$-algebra $\cB(\Theta)$ on $\Theta$ and the Borel $\sigma$-algebra on $\RRR^{d}$, and $G(\cdot,z)$ is lower semicontinuous for every $z\in\RRR^{d}$.
\item[(A2)] 
There is some strictly positive $\MP^{Z}$-integrable mapping $\xi:\RRR^{d}\rightarrow\RRR$ such that 
$$
\sup\limits_{\theta\in\Theta}|G(\theta,z)|\leq\xi(z)\quad\mbox{for}~z\in\RRR^{d}.
$$ 
\end{itemize}
Note that under these assumptions the optimization problems \eqref{optimization risk neutral} and \eqref{SAAriskneutral} are well defined with finite optimal values.
\begin{proposition}
\label{MessbarkeitLoesbarkeit}
Under (A1), (A2) the following statements are valid.
\begin{itemize}
\item [1)] $\inf\limits_{\theta\in\Theta}\frac{1}{n}\sum\limits_{j=1}^{n}G(\theta,Z_{j}) - \inf\limits_{\theta\in\Theta}\EEE[G(\theta,Z_{1})]$ is a random variable on $\OFP$ for every $n\in\NNN$.
\item [2)] There exists a solution of \eqref{optimization risk neutral}, and the set of solutions is compact.
\end{itemize}
\end{proposition}
The proof may be found in Subsection \ref{Beweis von Proposition MessbarkeitLoesbarkeit}.
\smallskip

In order to develop nonasymptotic confidence intervals on the optimal value of \eqref{optimization risk neutral} the authors in \cite{GuiguesEtAl2017} provide specific separate upper estimates for the deviation probabilities
\begin{align}
&\label{deviation unten}
\MP\Big(\Big\{\inf\limits_{\theta\in\Theta}~\frac{1}{n}~\sum_{j=1}^{n}G(\theta,Z_{j}) - \inf\limits_{\theta\in\Theta}\EEE\big[G(\theta,Z_{1})\big]\leq -\varepsilon/\sqrt{n}\Big\}\Big)\quad(n\in\NNN,\varepsilon > 0),\\
&\label{deviation oben}
\MP\Big(\Big\{\inf\limits_{\theta\in\Theta}~\frac{1}{n}~\sum_{j=1}^{n}G(\theta,Z_{j}) - \inf\limits_{\theta\in\Theta}\EEE\big[G(\theta,Z_{1})\big]\geq\varepsilon/\sqrt{n}\Big\}\Big)\quad(n\in\NNN,\varepsilon > 0).
\end{align}
They assume $G$ to be convex in $\theta$ such that the goal function of \eqref{optimization risk neutral} is differentiable, and they also impose conditions on the tail behaviour of the random variables $G(\theta,Z)$. Avoiding such regularity conditions, in \cite{Kraetschmer2022} upper estimations has been derived for deviations probabilities
\begin{equation}
\label{deviation beidseitig}
\MP\Big(\Big\{\big|\inf_{\theta\in\Theta}\frac{1}{n}\sum_{i=1}^{n}G(\theta,Z_{j}) - \inf_{\theta\in\Theta}\EEE[G(\theta,Z)]\big|\geq\varepsilon\Big\}\Big)\quad(n\in\NNN,\varepsilon > 0).
\end{equation}
There, no further restrictions beyond property (A2) are imposed on the tail behaviour of the random variables $G(\theta,Z)$. Moreover, the analytical requirements on paths of $G$ are replaced with some specific condition on the function class $\{G(\theta,\cdot)\mid\theta\in\Theta\}$ which we shall explain in more detail soon in this section.  
\smallskip

The results in \cite{Kraetschmer2022} 
already imply
$$
n^{a}\Big[\inf_{\theta\in\Theta}\int_{\RRR}t ~d\hat{F}_{n,\theta}(t) - \inf_{\theta\in\Theta}\EEE\big[G(\theta,Z)\big]\Big]\underset{n\to\infty}{\to} 0\quad\mbox{in probability}\quad\mbox{for}~a\in [0,1/2[
$$
(see \cite[Theorem 2.2]{Kraetschmer2022}). Moreover, the sequence \eqref{sequence} is uniformly tight, i.e. relatively compact w.r.t. the topology of weak covergence (see \cite[Theorem 2.5]{Kraetschmer2022}).
\smallskip

Throughout this section we want to complete the results from \cite{Kraetschmer2022} by some criterion to ensure weak convergence of the sequence \eqref{sequence}, and to find its asymptotic distributions. In contrast to the literature on the first order asymptotics of the SAA method we do not want to impose analytical properties for the objective $G$ like e.g. continuity convexity in the parameter $\theta$. Instead, as in \cite{Kraetschmer2022}, we suggest a condition which makes the function class $\FFF^{\Theta} := \{G(\theta,\cdot)\mid \theta\in\Theta\}$ small in some sense. Convenient ways to express this idea may be provided by general devices from empirical process theory which are based on covering numbers for classes of Borel measurable mappings from $\RRR^{d}$ into $\RRR$ w.r.t. $L^{p}$-norms.  To recall these concepts adapted to our situation, let us fix any nonvoid set $\FFF$ of Borel measurable mappings from $\RRR^{d}$ into $\RRR$ and any probability measure $\MQ$ on $\cB(\RRR^{d})$ with metric $d_{\MQ,p}$ induced by the $L^{p}$-norm $\|\cdot\|_{\MQ,p}$ for $p\in [1,\infty[$.  
\begin{itemize}
\item \textit{Covering numbers for $\FFF$}\\
We use $N\big(\eta,\FFF,L^{p}(\MQ)\big)$ to denote the minimal number to cover $\FFF$ by closed $d_{\MQ,p}$-balls of radius $\eta > 0$ with centers in $\FFF$. We define $N\big(\eta,\FFF,L^{p}(\MQ)\big) := \infty$ if no finite cover is available. 
\item An \textit{envelope} of $\FFF$ is defined to mean some Borel measurable mapping $C_{\FFF}:\RRR^{d}\rightarrow\RRR$ satisfying $\sup_{h\in\FFF}|h|\leq C_{\FFF}$. If an envelope $C_{\FFF}$ has strictly positive outcomes, we shall speak of a \textit{positive envelope}.
\item $\cM_{\textrm{\tiny fin}}$ denotes the set of all probability measures on $\cB(\RRR^{d})$ with finite support.
\end{itemize}
Usually, upper estimations of covering numbers are used instead of exact calculations.
\medskip

For abbreviation let us introduce for a class $\FFF$ of Borel measurable functions from $\RRR^{d}$ into $\RRR$ with arbitrary positive envelope $C_{\FFF}$ of $\FFF$ the following notation
\begin{align}
&
\label{Entropie-Integral I}
\overline{J}(\FFF,C_{\FFF},\delta) := \int_{0}^{\delta}\sup_{\MQ\in \cM_{\textrm{\tiny fin}}}\sqrt{\ln\big( N\big(\varepsilon~\|C_{\FFF}\|_{\MQ,2},\FFF,L^{2}(\MQ)\big)\big)}~d\varepsilon.
\end{align}
As we shall see, the finiteness of some integral $\overline{J}(\FFF^{\Theta},C_{\FFF^{\Theta}},1)$ is already sufficient for our purposes. Henceforth we shall restrict considerations to ``small'' classes $\FFF^{\Theta}$ in the sense that $\overline{J}(\FFF^{\Theta},C_{\FFF^{\Theta}},1)$ is finite for some positive square $\MP^{Z}$-integrable envelope $C_{\FFF^{\Theta}}$ of $\FFF^{\Theta}$.
\smallskip

In the following result concerning the asymptotic distribution of the sequence \eqref{sequence} we shall use the symbol $\argmin_{\textrm{\tiny RN}}$ to denote the set of minimizers of the problem \eqref{optimization risk neutral} which is nonvoid and compact by Proposition \ref{MessbarkeitLoesbarkeit}. Furthermore, we shall endow $\Theta$ with the alternative semimetric $\overline{d}_{\Theta}$ defined by $\overline{d}_{\Theta}(\theta,\vartheta) := \sqrt{\vari\big(G(\theta,Z_{1}) - G(\vartheta,Z_{1})\big)}$. 
\begin{theorem}
\label{asymptoticDistributionRiskNeutral}
Let (A1), (A2) be fulfilled, where the mapping $\xi$ from (A2) is square $\MP^{Z}$-integrable.  
Using notation \eqref{Entropie-Integral I}, if $\overline{J}(\FFF^{\Theta},\xi,1)$ is finite, then $\overline{d}_{\Theta}$ is totally bounded and there exists some centered Gaussian process $\mathfrak{G} = (\mathfrak{G}_{\theta})_{\theta\in\Theta}$ 
such that 
the sequence \eqref{sequence} converges weakly to 
$
\sup\limits_{\theta\in\argmin_{\textrm{\tiny RN}}}\mathfrak{G}_{\theta}.
$
This Gaussian process has uniformly continuous paths w.r.t. $\overline{d}_{\Theta}$ and satisfies 
$$
\EEE\big[\mathfrak{G}_{\theta}\cdot \mathfrak{G}_{\vartheta}\big] = 
\covi
\big(G(\theta,Z_{1}), G(\vartheta,Z_{1})\big)\quad\mbox{for}~\theta, \vartheta\in\Theta.
$$
In particular, if the optimization problem \eqref{optimization risk neutral} has a unique solution $\theta^{*}$, under the given assumptions the sequence \eqref{sequence} converges weakly to some centered normally distributed random variable with variance $\vari\big(G(\theta^{*},Z_{1})\big)$. 
\end{theorem}
The proof is delegated to Subsection \ref{Beweis von Hauptresultat}.
\medskip

The result on the asymptotic distribution crucially requires $\overline{J}(\FFF^{\Theta},\xi,1)$ to be finite. This property is always satisfied if the involved covering numbers have polynomial rates. Indeed this relies on the observation, that by using change of variable formula several times along with integration by parts, we obtain
\begin{equation}
\label{Integralabschaetzung}
\int_{0}^{1}\sqrt{v\ln(K/\varepsilon)}~d\varepsilon\leq 2\sqrt{v \ln(K)}\quad\mbox{for}~v\geq 1, K\geq e.
\end{equation}
Inequality \eqref{Integralabschaetzung} may be applied 
if there exist $K\geq e, v\geq 1$ such that the following condition is satisfied
\begin{equation}
\label{polynomialRate}
N\big(\varepsilon~\|C_{\FFF^{\Theta}}\|_{\MQ,2},\FFF^{\Theta},L^{2}(\MQ)\big)\big)\leq (K/\varepsilon)^{v}\quad\mbox{for}~\MQ\in\cM_{\textrm{\tiny fin}}\quad\mbox{and}~\varepsilon\in ]0,1[.
\end{equation}
Prominent examples satisfying \eqref{polynomialRate} are provided by so called \textit{VC-subgraph classes} (see e.g. \cite{vanderVaartWellner1996}).
\par
In \cite{Kraetschmer2022} explicit upper estimates of the terms $\overline{J}(\FFF^{\Theta}, \xi,\delta)$ have been derived for objectives $G$ satisfying specific analytical properties. The line of reasoning there is to show property \eqref{polynomialRate} and then to utilize \eqref{Integralabschaetzung} (see Propositions 2.6, 2.8 and their proofs in \cite{Kraetschmer2022}). Let us recall these specializations.
\smallskip

Denoting the Euclidean metric on $\RRR^{m}$ by $d_{m,2}$, the first one is built on the following condition.
\begin{enumerate}
\item[(H)] There exist some $\beta\in ]0,1]$ and a square $\MP^{Z}$-integrable strictly positive mapping $C:\RRR^{d}\rightarrow ]0,\infty[$ such that 
$$
\big|G(\theta,z) - G(\vartheta,z)\big|\leq C(z)~d_{m,2}(\theta,\vartheta)^{\beta}\quad\mbox{for}~z\in\RRR^{d}, \theta, \vartheta\in\Theta.
$$
\end{enumerate}
Property (H) simplifies Theorem \ref{asymptoticDistributionRiskNeutral}.
\begin{example}
\label{Hoelder-Bedingung auxiliary}
Let $\Theta$ be compact, let condition (H) be fulfilled with $\beta\in ]0,1]$, and let $G(\theta,\cdot)$ be Borel measurable for $\theta\in\Theta$. Furthermore, $G(\overline{\theta},\cdot)$ is assumed to be square $\MP^{Z}$-integrable for some $\overline{\theta}\in\Theta$.
\par
First of all (A1) is satisfied (see e.g. \cite[Lemma 6.7.3]{Pfanzagl1994}). Secondly, we are in the position to invoke Proposition 2.6 from \cite{Kraetschmer2022}. Hence, with 
$\Delta(\Theta)$ denoting the diameter of $\Theta$ w.r.t. $d_{m,2}$, the mapping $\xi := C\cdot \Delta(\Theta)^{\beta} + |G(\overline{\theta},\cdot)|$ is square $\MP^{Z}$-integrable satisfying property (A2), and $\overline{J}(\FFF^{\Theta},\xi,1)$ is finite with certain explicit upper estimates. 
In particular, Theorem \ref{asymptoticDistributionRiskNeutral} may be applied directly, and all the statements there carry over immediately. This extends the classical result on the asymptotic distributions of the sequence \eqref{sequence}, where condition (H) with $\beta = 1$ is imposed (see \cite{ShapiroEtAl}).
\end{example}
\smallskip

In the second special case from \cite{Kraetschmer2022} objectives $G$ have been considered having the following kind of structure of piecewise H\"older continuity. 
\begin{enumerate}
\item [(PH)] $
G(\theta,z) = \sum\limits_{i = 1}^{r}\eins_{\bigcap_{l=1}^{s_{i}}\{\Lambda_{il}(\theta, \cdot) + a^{i}_{l}\in I_{il}\}}(z)\cdot G^{i}(\theta,z),
$ 
where
\begin{itemize}
\item $r, s_{1},\dots,s_{r}\in\NNN$,
\item $G^{i}$ satisfies (A1), and (H) with $\beta_{i}\in ]0,1]$ as well as strictly positive square $\MP^{Z}$-integrable $C_{i}:\RRR^{d}\rightarrow\RRR$ for $i\in\{1,\ldots,r\}$,
\item 
$\Lambda_{il}:\RRR^{m}\times\RRR^{d}\rightarrow\RRR$ Borel measurable with $\Lambda_{il}(\cdot,z)$ affine linear for $z\in\RRR^{d}$ ($i\in\{1,\ldots,r\}$, $l\in\{1,\ldots,s_i\}$),
\item $a^{i}_{l}\in\RRR$ for $i\in\{1,\dots,r\}, l\in\{1,\dots,s_{i}\}$,
\item $I_{il} = ]0,\infty[$ or $I_{il} = [0,\infty[$ for $i\in\{1,\dots,r\}$ and $l\in\{1,\dots,s_{i}\}$,
\item The set
$$
\Big\{\bigcap\limits_{l=1}^{s_{i}}\big\{\Lambda_{il}(\theta, \cdot) + a^{i}_{l}\in I_{il}\}\mid i\in\{1,\ldots,r\}\big\}\Big\}
$$
is a partition of $\RRR^{d}$.
\end{itemize}
\end{enumerate}
Note that $G$ satisfying condition (PH) does not have continuity or convexity in $\theta$ in advance.
\par
In two stage mixed-integer programs the goal functions typically may be represented by (PH) if the random vector $Z$ has compact support (see \cite[p. 121]{EichhornRoemisch2007} together with \cite{Kraetschmer2022}). Within this special situation with compact $\Theta$ the authors in \cite{EichhornRoemisch2007} derive the same asymptotic distributions for the sequence \eqref{sequence} as in Theorem \ref{asymptoticDistributionRiskNeutral}. Their line of reasoning is based upon the corresponding representation (PH) of $G$, and it relies also on finiteness of the integrals $\overline{J}(\FFF^{\Theta}, C_{\FFF^{\Theta}},\delta)$. We may extend their result to general objectives $G$ having representation (PH), because under this condition the application of Theorem \ref{asymptoticDistributionRiskNeutral} is quite immediate.
\begin{example}
\label{startingpoint}
Let $\Theta$ be compact, and let $G$ be lower semicontinuous in $\theta$ having representation (PH) with mappings $G^{i}, C_{i}$ and $\beta_{i}\in ]0,1]$ for $i\in\{1,\ldots,r\}$. Moreover, $G^{1}(\overline{\theta},\cdot),\ldots,G^{r}(\overline{\theta},\cdot)$ are supposed to be square $\MP^{Z}$-integrable for some $\overline{\theta}\in\Theta$. The notation $\Delta(\Theta)$ stands for the diameter of $\Theta$ w.r.t. the Euclidean metric on $\RRR^{m}$.
\par
According to Proposition 2.8 from \cite{Kraetschmer2022} requirement (A1) is fulfilled, and $$
\xi := \sum_{i=1}^{r}\big[\Delta(\Theta)^{\beta_{i}}~C_{i} + |G^{i}(\overline{\theta},\cdot)|\big]
$$ 
is square $\MP^{Z}$-integrable, satisfying (A2) and $\overline{J}(\FFF^{\Theta},\xi,1) < \infty$. Also explicit upper estimations of $\overline{J}(\FFF^{\Theta},\xi,1)$ are provided there. Hence all requirements of Theorem \ref{asymptoticDistributionRiskNeutral} are met.
\end{example}
\begin{remark}
\label{confidence interval I}
Let the assumptions of Theorem \ref{asymptoticDistributionRiskNeutral} be fulfilled. If optimization \eqref{optimization risk neutral} has a unique solution $\theta^{*}$ we may utilize the result of Theorem \ref{asymptoticDistributionRiskNeutral} to construct asymptotic confidence intervals on the optimal value $\EEE[G(\theta^{*},Z)]$ of \eqref{optimization risk neutral} in the following way. Choose for level $\beta\in ]0,1[$ real numbers $a,b > 0$ such that $\Phi_{N(0,1)}(a) + \Phi_{N(0,1)}(b) \geq 2 -\beta$ holds, where $\Phi_{N(0,1)}$ denotes the distribution function of the standard normal distribution. Then for every positive upper estimate 
$L\geq \sqrt{\vari\big(G(\theta^{*},Z_{1})\big)}$ we may define by 
$$
I^{n}_{a,b, L} := \Big[\inf_{\theta\in\Theta}\frac{1}{n}\sum_{j=1}^{n}G(\theta,Z_{j}) - \frac{L~a}{\sqrt{n}}~,~\inf_{\theta\in\Theta}\frac{1}{n}\sum_{j=1}^{n}G(\theta,Z_{j}) + \frac{L~b}{\sqrt{n}}\Big]
$$
a sequence $\big(I^{n}_{a,b, L}\big)_{n\in\NNN}$ of confidence intervals on $\EEE[G(\theta^{*},Z)]$ which fulfills 
$$
\liminf_{n\to\infty}\MP\big(\big\{\EEE[G(\theta^{*},Z)]\in I^{n}_{a,b, L}\big\}\big)\geq 1 - \beta.
$$
These confidence intervals may be considered as an alternative of the nonasymptotic confidence intervals which may be built directly on the upper estimates for the deviation probabilities \eqref{deviation beidseitig} 
derived in Theorem 2.2 from \cite{Kraetschmer2022}. It should be emphasized that these both ways to find confidence intervals do not require in advance path properties of the objective $G$ like continuity or convexity. In particular these methods may be used in the case that objectives have representation (PH), e.g. in two stage mixed-integer programs. 
\par
In \cite{GuiguesEtAl2017} the authors develop nonasymptotic confidence intervals which are based on specific separate upper estimates for the deviation probabilities \eqref{deviation unten} and \eqref{deviation oben}.
The special feature of their suggestion is that the confidence intervals are independent of the dimension of the parameters (see \cite[Discussion 2.1.3, (3)]{GuiguesEtAl2017}). However, $G$ has to be convex in the parameter.
\end{remark}
\section{First order asymptotics under absolute semideviations}
\label{FOA absolute semideviations}
Let $L^1(\Omega,\cF,\pr)$ denote the usual $L^{1}$-space on $\OFP$, where we tacitely identify random variables which are different on $\MP$-null sets only. 
\par
We want to study the risk averse stochastic program (\ref{optimization risk averse}), where  
in the objective the functional $\rho$ is an \textit{absolute semideviation}. This means that for $a\in ]0,1]$ the functional $\rho = \rho_{1,a}$ is defined as follows
$$
\rho_{1,a}:L^{1}\OFP\rightarrow\RRR,~X\mapsto\EEE[X] + a~ \EEE\big[\big(X - \EEE[X]\big)^{+}\big].
$$
It is well-known that absolute semideviations are increasing w.r.t. the increasing convex order (cf. e.g. \cite[Theorem 6.51 along with Example 6.23 and Proposition 6.8]{ShapiroEtAl}). They are also distribution-invariant so that we may define the associated functional $\cR_{\rho_{1,a}}$ on the set of distributions functions of random variables with first absolute moments. The aim of this section is the optimization problem 
\begin{equation}
\label{optimization semideviation}
\inf_{\theta\in\Theta}\cR_{\rho_{1,a}}\big(F_{\theta}\big),
\end{equation}
where $F_{\theta}$ stands for the distribution function of $G(\theta,Z)$ for 
$\theta\in\Theta$. The set of minimizers of this problem will be denoted by $\argmin\cR_{\rho_{1,a}}$. 

Introducing the notation 
\begin{equation}
\label{Hilfsgoals}
G_{1}:\Theta\times\RRR^{d}\rightarrow\RRR,~(\theta,z)\mapsto \big(G(\theta,z) - \EEE[G(\theta,Z_{1})]\big)^{+}, 
\end{equation}
we may describe this optimization also in the following way
\begin{equation}
\label{optimization general III}
\inf_{\theta\in\Theta}\cR_{\rho_{1,a}}\big(F_{\theta}\big) = 
\inf_{\theta\in\Theta}\left\{\EEE[G(\theta,Z_{1})] + a~\EEE[G_{1}(\theta,Z_{1})]\right\}
\end{equation}
The stochastic objective of the approximative problem according to the SAA method has the following representation.
\begin{equation}
\label{SAA upper semideviations}
\cR_{\rho_{1,a}}\big(\hat{F}_{n,\theta}\big) = 
\frac{1}{n}\sum_{j=1}^{n}G(\theta,Z_{j}) + a~\frac{1}{n}\sum_{j=1}^{n}\big[G(\theta,Z_{j}) - \frac{1}{n}\sum_{i=1}^{n}G(\theta,Z_{i})\big]^{+}
\end{equation}

\smallskip

Let (A1), (A2) be fulfilled, and let $\overline{J}(\FFF^{\Theta}, \xi,1) < \infty$, where $\xi$ is from (A2). Then under some minor additional regularity conditions on $G$ we already know that
$$
n^{a}~\Big[\inf\limits_{\theta\in\Theta}\cR_{\rho_{1,a}}\big(\hat{F}_{n,\theta}\big) - \inf\limits_{\theta\in\Theta}\cR_{\rho_{1,a}}\big(F_{\theta}\big)\Big]\underset{n\to\infty}{\to} 0\quad\mbox{in probability}
$$ 
holds for $a\in ]0,1/2[$ (see \cite[Theorem 3.5]{Kraetschmer2022}). Moreover, also by Theorem 3.5 from \cite{Kraetschmer2022}, the sequence
\begin{equation}
\label{sequence object}
\Big(\sqrt{n}~\Big[\inf\limits_{\theta\in\Theta}\cR_{\rho_{1,a}}\big(\hat{F}_{n,\theta}\big) - \inf\limits_{\theta\in\Theta}\cR_{\rho_{1,a}}\big(F_{\theta}\big)\Big]\Big)_{n\in\NNN}
\end{equation}
is uniformly tight. The aim of this section is to find asymptotic distributions of this sequence. The starting point is the following observation from \eqref{optimization general III}
\begin{equation}
\label{startingobservation}
\rho_{1,a}\big(G(\theta,Z_{1})\big) = \EEE[\widehat{G}_{1,a}(\theta,Z_{1})] \quad\mbox{for}~\theta\in\Theta, 
\end{equation}
where, setting $\overline{F}_{\theta} := 1- F_{\theta}$, the mapping $\widehat{G}_{1,a}:\Theta\times\RRR^{d}\rightarrow\RRR$ is defined by
\begin{align*}
\widehat{G}_{1,a}(\theta,z)
= \big[1- a\overline{F}_{\theta}\big(\EEE[G(\theta,Z_{1})]\big)\big]~G(\theta,z) + a \EEE[G(\theta,Z_{1})]\overline{F}_{\theta}\big(\EEE[G(\theta,Z_{1})]\big) + a G_{1}(\theta,z).
\end{align*}
Then the key is to show that sequence \eqref{sequence object} has the same asymptotic distribution as the following sequence
\begin{equation}
\label{SAA upper semideviations auxiliary}
\Big(\sqrt{n}\Big[\inf_{\theta\in\Theta}\frac{1}{n}\sum_{j=1}^{n}\widehat{G}_{1,a}(\theta,Z_{j}) - \inf_{\theta\in\Theta}\EEE\big[\widehat{G}_{1,a}(\theta,Z_{1})\big]\Big]\Big)_{n\in\NNN}
\end{equation}
if one of these sequences converges weakly. In this case we may apply Theorem \ref{asymptoticDistributionRiskNeutral} to the 
sequence \eqref{SAA upper semideviations auxiliary} to derive the asymptotic distribution of sequence \eqref{sequence object}.
\smallskip 

The investigations are based on the following mild continuity requirement for the objective $G$.
\begin{itemize}
\item [(A3)] $G(\theta_{n},\cdot)\to G(\theta,\cdot)$ in $\MP^{Z}$-probability whenever $\theta_{n}\to\theta$ w.r.t. the Euclidean metric.
\end{itemize}
Requirement (A3) implies some useful convergence property in continuity points of the distribution functions $F_{\theta}$.
\begin{lemma}
\label{continuity convergence}
Let $t\in\RRR$ be a continuity point of $F_{\theta}$ for some $\theta\in\Theta$, and let 
(A3) be fulfilled. If $\theta_{n}\to\theta$ and $t_{n}\to t$, then $F_{\theta_{n}}(t_{n})\to F_{\theta}(t)$. 
\end{lemma}
\begin{proof}
Fix any $\varepsilon > 0$. Then there is some $n_{0}\in\NNN$ such that $|t_{n} - t|\leq\varepsilon/2$ for $n\in\NNN$ with $n\geq n_{0}$. Furthermore
\begin{align*}
&
F_{\theta_{n}}(t_{n})\leq F_{\theta_{n}}(t + \varepsilon/2)\leq F_{\theta}(t + \varepsilon) + \MP\big(\big\{|G(\theta_{n},Z_{1}) - G(\theta,Z_{1})| > \varepsilon/2\big\}\big)\\
&
F_{\theta_{n}}(t_{n})\geq F_{\theta_{n}}(t - \varepsilon/2)\geq  \MP\big(\big\{G(\theta,Z_{1}) \leq t-\varepsilon,|G(\theta_{n},Z_{1}) - G(\theta,Z_{1})|\leq \varepsilon/2\big\}\big)\\
&\qquad\quad
\geq 
F_{\theta}(t - \varepsilon) - \MP\big(\big\{|G(\theta_{n},Z_{1}) - G(\theta,Z_{1})| > \varepsilon/2\big\}\big)
\end{align*}
for $n\in\NNN$ with $n\geq n_{0}$. Hence by (A3)
\begin{align*}
\liminf_{n\to\infty}F_{\theta_{n}}(t_{n})\geq F_{\theta}(t-\varepsilon)\quad\mbox{and}\quad\limsup_{n\to\infty}F_{\theta_{n}}(t_{n})\leq F_{\theta}(t+\varepsilon).
\end{align*}
The statement may be derived immediately from continuity of $F_{\theta}$ at $t$.
\end{proof}
The following result is the providing step to show that the sequences \eqref{sequence object} and \eqref{SAA upper semideviations auxiliary} have identical asymptotic distributions if \eqref{SAA upper semideviations auxiliary} converges weakly.
\begin{proposition}
\label{missing link I}
Let (A1) - (A3) be fulfilled, where the mapping $\xi$ from (A2) is square $\MP^{Z}$-integrable.  
Furthermore the distribution function $F_{\theta}$ is continuous at $\EEE[G(\theta,Z_{1})]$ for $\theta\in\Theta$. Using notation \eqref{Entropie-Integral I}, if $\overline{J}(\FFF^{\Theta},\xi,1)$ is finite, then
$$
\lim_{n\to\infty}\MP^{*}\Big(\Big\{\big|\sqrt{n}\inf_{\theta\in\Theta}\cR_{\rho_{1,a}}\big(\hat{F}_{n,\theta}\big) - \sqrt{n}\inf_{\theta\in\Theta}\frac{1}{n}\sum_{j=1}^{n}\widehat{G}_{1,a}(\theta,Z_{j})\big| > \varepsilon\Big\}\Big) = 0\quad\mbox{for}~\varepsilon > 0,
$$
where $\MP^{*}$ stands for the outer probability of $\MP$. 
\end{proposition}
The proof of Proposition \ref{missing link I} may be found in Subsection \ref{Beweis missing link I}.
\smallskip

Now, combining Proposition \ref{missing link I} with Theorem \ref{asymptoticDistributionRiskNeutral}, we may derive our result on the first order asymptotics of the SAA under absolute semideviations.
\begin{theorem}
\label{first asymptotics absolute semideviations}
Let $a\in ]0,1]$, let (A1) - (A3) be fulfilled, where the mapping $\xi$ from condition 
(A2) is $\MP^{Z}$-integrable of order $4$. Furthermore the distribution function $F_{\theta}$ is continuous at $\EEE[G(\theta,Z_{1})]\}$ for $\theta\in\Theta$. Using notation \eqref{Entropie-Integral I}, if $\overline{J}(\FFF^{\Theta},\xi,1)$ is finite, then the following statements hold. 
\begin{itemize}
\item [1)] $\sqrt{n}\big[\inf\limits_{\theta\in\Theta}\cR_{\rho_{1,a}}\big(\hat{F}_{n,\theta}\big) - \inf\limits_{\theta\in\Theta}\cR_{1,a}\big(F_{\theta}\big)\big]$ is a random variable on $\OFP$ for $n\in\NNN$.
\item [2)] The set $\argmin\cR_{\rho_{1,a}}$ is nonvoid and compact.
\item [3)]
There exists some centered Gaussian process $\widehat{\mathfrak{G}} = (\widehat{\mathfrak{G}}_{\theta})_{\theta\in\Theta}$ 
such that 
the sequence 
$$
\Big(\sqrt{n}\big[\inf_{\theta\in\Theta}\cR_{\rho_{1,a}}\big(\hat{F}_{n,\theta}\big) - \inf_{\theta\in\Theta}\cR_{1,a}\big(F_{\theta}\big)\big]\Big)_{n\in\NNN}
$$
converges weakly to 
$
\sup\limits_{\theta\in\argmin\cR_{\rho_{1,a}}}\widehat{\mathfrak{G}}_{\theta}.
$
This Gaussian process has uniformly continuous paths w.r.t. Euclidean metric and satisfies 
$$
\EEE\big[\widehat{\mathfrak{G}}_{\theta}\cdot \widehat{\mathfrak{G}}_{\vartheta}\big] = 
\covi
\big(\widehat{G}_{1,a}(\theta,Z_{1}), \widehat{G}_{1,a}(\vartheta,Z_{1})\big)\quad\mbox{for}~\theta, \vartheta\in\Theta.
$$
In particular, if the optimization problem \eqref{optimization general III} has a unique solution $\theta^{*}$, then under the given assumptions 
we have weak convergence to some centered normally distributed random variable with variance $\vari\big(\widehat{G}_{1,a}(\theta^{*},Z_{1})\big)$. 
\end{itemize}
\end{theorem}
\begin{proof}
Firstly, by (A2) the genuine optimization problem and its SAA counterpart have finite optimal values. Furthermore 
$$
\big\{\omega\in\Omega\mid \inf_{\theta\in\Theta}\cR_{\rho_{1,a}}\big(\hat{F}_{n,\theta}\big)_{|\omega} < t\big\} = \textrm{Pr}_{\Omega}\Big(\big\{(\theta,\omega)\in\Theta\times\Omega\mid \cR_{\rho_{1,a}}\big(\hat{F}_{n,\theta}\big)_{|\omega} < t\big\}\Big)\quad\mbox{for}~t\in\RRR,
$$
where $\textrm{Pr}_{\Omega}$ denotes the standard projection from $\Theta\times\Omega$ onto $\Omega$. In view of \eqref{SAA upper semideviations} the set $\big\{(\theta,\omega)\in\Theta\times\Omega\mid \cR_{\rho_{1,a}}\big(\hat{F}_{n,\theta}\big)_{|\omega} < t\big\}$ belongs to $\cB(\Theta)\otimes\cF$ for every $t\in\RRR$ due to (A1). Since $\Theta$ is a Polish space, and since $\OFP$ is complete, we end up with $\big\{\omega\in\Omega\mid \inf_{\theta\in\Theta}\cR_{\rho_{1,a}}\big(\hat{F}_{n,\theta}\big)_{|\omega} < t\big\}\in\cF$ (see \cite[Proposition 8.4.4]{Cohn1980}). This shows statement 1).
\smallskip

Secondly, $\EEE[\widehat{G}_{1,a}(\theta,Z_{1})] = \rho_{1,a}\big(G(\theta,Z_{1})\big)$ holds for any $\theta\in\Theta$. Hence in view of Proposition \ref{missing link I} along with a version of Slutsky's lemma (\cite[Lemma 1.10.2]{vanderVaartWellner1996}) it remains to show that the objective $\widehat{G}_{1,a}$ meets the requirements of Proposition \ref{MessbarkeitLoesbarkeit} and Theorem \ref{asymptoticDistributionRiskNeutral}.
\medskip

In Lemma \ref{basic observations} below it will be shown that the 
mapping $\theta\mapsto\EEE[G(\theta,Z_{1})]$ on $\Theta$ is continuous under (A1) - (A3). 
This implies by Lemma \ref{continuity convergence} the continuity of the mappings $\theta\mapsto \overline{F}_{\theta}\big(\EEE[G(\theta,Z_{1})]\big)$ and $\theta\mapsto \EEE[G(\theta,Z_{1})]~\overline{F}_{\theta}\big(\EEE[G(\theta,Z_{1})]\big)$ on $\Theta$ because each $F_{\theta}$ is assumed to be continuous at $\EEE[G(\theta,Z_{1})]$. Since in addition $\big[1 - a \overline{F}_{\theta}\big(\EEE[G(\theta,Z_{1})]\big)\big]$ is nonnegative for $\theta$ and $(\cdot)^{+}$ is nondecreasing, the objective $\widehat{G}_{1,a}$ is lower semicontinuous in $\theta$ due to (A1). Property (A1) also implies that $\widehat{G}_{1,a}$ is measurable w.r.t. $\cB(\Theta)\otimes\cB(\RRR^{d})$.
\smallskip

For a nonnvoid bounded subset $\cK$ of $\RRR$ we denote by $\overline{\FFF}^{\cK}$ the set of all constant mappings on $\RRR^{d}$ with outcomes in $\cK$. We shall use notation $N(\varepsilon,\cK,|\cdot|)$ for the minimal number to cover $\cK$ by intervals of the form $[a - \varepsilon,a+ \varepsilon]$, where $\varepsilon > 0$ and $a\in\cK$. Then
\begin{align*}
N\big(\varepsilon,\overline{\FFF}^{\cK},L^{2}(\MQ)\big)\leq N(\varepsilon,\cK,|\cdot|)\leq \frac{2 (\sup\cK - \inf\cK)}{\varepsilon}\quad\mbox{for}~\MQ\in \cM_{\textrm{\tiny fin}},~\varepsilon > 0.
\end{align*}
In particular, using notation \eqref{Entropie-Integral I}, we have $\overline{J}(\overline{\FFF}^{\cK},c,1) < \infty$ for every positive, constant envelope $c$ of $\overline{\FFF}^{\cK}$. So we have finiteness of the terms $\overline{J}(\overline{\FFF}^{\cK_{1}},c_{1},1)$ and $\overline{J}(\overline{\FFF}^{\cK_{2}},c_{2},1)$, where  
$\cK_{1} := \big\{1 - a \overline{F}_{\theta}\big(\EEE[G(\theta,Z_{1})]\big)\mid \theta\in\Theta\big\}$, $\cK_{2} := \big\{a ~\EEE[G(\theta,Z_{1})]~\overline{F}_{\theta}\big(\EEE[G(\theta,Z_{1})]\big)\mid \theta\in\Theta\big\}$ and $c_{1} := 1$, $c_{2} :=\EEE[\xi(Z_{1})]$. Moreover, since $\xi$ is $\MP^{Z}$-integrable of order $4$, we obtain by Lemma 3.4 from \cite{Kraetschmer2022} that there exists some square $\MP^{Z}$-integrable positive envelope $\xi_{1,a}$ of $\overline{\FFF}^{\Theta,1}_{a} := \{a~G_{1}(\theta,\cdot)\mid \theta\in\Theta\}$ with finite $\overline{J}(\overline{\FFF}_{a}^{\Theta,1},\xi_{1,a},1)$. Now, invoking Lemma 9.14 and Theorem 9.15 both from \cite{Kosorok2008}, and recalling $\overline{J}(\FFF^{\Theta},\xi,1) < \infty$, the mapping $\xi + \EEE[\xi(Z_{1})] + \xi_{1,a}$ is a square $\MP^{Z}$-integrable positive envelope of $\widehat{\FFF}_{1,a} := \{\widehat{G}_{1,a}(\theta,\cdot)\mid\theta\in\Theta\}$ with finite $\overline{J}\big(\widehat{\FFF}_{1,a},\xi + \EEE[\xi(Z_{1})]+\xi_{1,a},1\big)$.
\medskip

Now, we are ready to apply Proposition \ref{MessbarkeitLoesbarkeit} and Theorem \ref{asymptoticDistributionRiskNeutral} to the objective $\widehat{G}_{1,a}$. Statement 2) follows immediately from Proposition \ref{MessbarkeitLoesbarkeit}. 
According to Theorem \ref{asymptoticDistributionRiskNeutral} we may find some centered Gaussian process $\widehat{\mathfrak{G}} = (\widehat{\mathfrak{G}}_{\theta})_{\theta\in\Theta}$ with the same covariances as in statement 3) such that 
the sequence \eqref{SAA upper semideviations auxiliary}
converges weakly to 
$
\sup\limits_{\theta\in\argmin\cR_{\rho_{1,a}}}\widehat{\mathfrak{G}}_{\theta}.
$
If the optimization problem \eqref{optimization general III} has a unique solution $\theta^{*}$, the limit is a centered normally distributed random variable with variance $\vari\big(\widehat{G}_{1,a}(\theta^{*},Z_{1})\big)$. 
\smallskip

It is also known from Theorem \ref{asymptoticDistributionRiskNeutral} that the process 
$\widehat{\frak{G}}$ has uniformly continuous paths w.r.t. the semimetric
$\overline{d}_{\Theta,1,a}$ defined by 
$\overline{d}_{\Theta,1,a}(\theta,\vartheta) = \sqrt{\vari\big(\widehat{G}_{1,a}(\theta,Z_{1}) - \widehat{G}_{1,a}(\vartheta,Z_{1})\big)}$.
By continuity of the three mappings 
$\theta\mapsto\EEE[G(\theta,Z_{1})]$, $\theta\mapsto \overline{F}_{\theta}\big(\EEE[G(\theta,Z_{1})]\big)$ as well as $\theta\mapsto \EEE[G(\theta,Z_{1})]~\overline{F}_{\theta}\big(\EEE[G(\theta,Z_{1})]\big)$ on $\Theta$ along with (A3) we may conclude convergence $\widehat{G}_{1,a}(\theta_{n},Z_{1})\to \widehat{G}_{1,a}(\theta,Z_{1})$ in probability for $\theta_{n}\to \theta$ w.r.t. the Euclidean metric.
Moreover in this case, the sequence $\big(\widehat{G}_{1,a}(\theta_{n},Z_{1})\big)_{n\in\NNN}$ is dominated by the square integrable random variable $\xi(Z_{1}) + \EEE[\xi(Z_{1})] + \xi_{1,a}(Z_{1})$. 
Then by Vitalis' theorem (see \cite[Proposition 21.4]{Bauer2001}) we end up with 
$\sqrt{\EEE[|\widehat{G}_{1,a}(\theta_{n},Z_{1}) -\widehat{G}_{1,a}(\theta,Z_{1})|^{2}]}\to 0$, and thus $\overline{d}_{\Theta,1,a}(\theta_{n},\theta)\to 0$. Therefore $\widehat{\frak{G}}$ has also continuous paths w.r.t. the Euclidean metric, and they are even uniformly continuous due to compactness of $\Theta$.
%
\end{proof}
\begin{remark}
Let $G(\theta,\cdot)$ be an $\MP^{Z}$-integrable random variable with distribution function $F_{\theta}$ being continuous at $\EEE[G(\theta,Z_{1})]$ for $\theta\in\Theta$. If $G$ either satisfies condition (H) or has representation (PH), then Example \ref{Hoelder-Bedingung auxiliary} or Example \ref{startingpoint} respectively provide constructions for proper positive envelopes $\xi$ of $\FFF^{\Theta}$ to apply Theorem \ref{first asymptotics absolute semideviations}.
\end{remark}
\begin{remark}
\label{confidence interval II}
Theorem \ref{first asymptotics absolute semideviations} offers a method to construct asymptotic confidence intervals on the optimal value of optimization \eqref{optimization semideviation} in the case that there is some unique solution $\theta^{*}$. The way is exactly the same one as we may find asymptotic confidence intervals on the optimal value of  \eqref{optimization risk neutral} according to Remark \ref{confidence interval I}. The only difference is the choice of the involved positive estimate $L$ which should satisfy $L\geq \sqrt{\vari\big(\widehat{G}_{1,a}(\theta^{*},Z_{1})\big)}$.
\par 
These asymptotic confidence intervals might be compared with nonasymptotic confidence intervals which are based on upper estimates for the deviation probabilities
$$
\MP\Big(\Big\{\big|\inf_{\theta\in\Theta}\cR_{\rho_{1,a}}\big(\hat{F}_{n,\theta}\big) - \inf_{\theta\in\Theta}\cR_{\rho_{1,a}}\big(F_{\theta}\big)\big|\geq\varepsilon\Big\}\Big)\quad(n\in\NNN,\varepsilon > 0)
$$
from Theorem 3.5 in \cite{Kraetschmer2022}. By Example \ref{startingpoint} both methods may be applied for objectives with representation (PH), e.g. in two stage mixed-integer programs.
\end{remark}
\section{First order asymptotics under divergence risk measures}
\label{generalized divergence risk measures}
Let 
denote by $L^p:=L^p(\Omega,\cF,\pr)$ the usual $L^{p}$-space on $\OFP$ ($p\in [0,\infty]$), where we tacitely identify random variables which are differ on $\MP$-null sets only. 
\par
We want to study the risk averse stochastic program (\ref{optimization risk averse}), where  
we shall focus on $\rho$ being
a divergence measure. For introduction, let us consider a lower semicontinuous convex mapping $\Phi: [0,\infty[\rightarrow [0,\infty]$ satisfying 
$\Phi(0) < \infty$, $\Phi(x_{0}) < \infty$ for some $x_{0} > 1,$ $\inf_{x\geq 0}\Phi(x) = 0,$ and the growth condition
$\lim_{x\to\infty}\frac{\Phi(x)}{x} = \infty.$ Its Fenchel-Legendre transform
$$
\Phi^{*}:\R\rightarrow \R\cup\{\infty\},~y\mapsto\sup_{x\geq 0}~\big(xy - \Phi(x)\big)
$$
is a finite nondecreasing convex function whose restriction $\Phi^{*}\bigr |_{[0,\infty[}$ to 
$[0,\infty[$ is a finite Young function, i.e.  a continuous nondecreasing and unbounded real valued mapping with $\Phi^{*}(0) = 0$ (cf. \cite[Lemma A.1]{BelomestnyKraetschmer2016}). Note also that the right-sided derivative $\Phi^{*'}_{+}$ of $\Phi^{*}$ is nonnegative and nondecreasing. We shall use  
$H^{\Phi^{*}}$ to denote the \textit{Orlicz heart} w.r.t. $\Phi^{*}\bigr |_{[0,\infty[}$ defined to mean the set of all random variables $X$ on $\OFP$ satisfying $\ex [\,\Phi^{*}(c|X|)\,]<\infty$ for all $c > 0$. Here we identify random variables which differ on $\MP$-null sets only.
\par
The Orlicz heart is known to be a vector space enclosing all $\MP$-essentially bounded random variables. Moreover,  
by Jensen's inequality all members of $H^{\Phi^{*}}$ are $\MP$-integrable.
For more on Orlicz hearts w.r.t. to Young functions the reader may consult \cite{EdgarSucheston1992}.
\medskip

We can define the following mapping
$$
\rho^{\Phi}(X)=\sup_{\overline{\MP}\in\cP_{\Phi}}\left(\ex_{\overline{\MP}}\left[X\right] - 
\ex\left[\Phi\left(\frac{d\overline{\MP}}{d\pr}\right)\right]\right)
$$
for all \(X\in H^{\Phi^{*}},\)  where $\cP_{\Phi},$ denotes the set of all probability measures $\overline{\MP}$ which are absolutely continuous w.r.t. $\pr$  such that $\Phi\left(\frac{d\overline{\MP}}{d\pr}\right)$ is $\pr-$integrable. Note that 
$\frac{d\overline{\MP}}{d\pr}~ X$ is $\pr-$integrable for every $\overline{\MP}\in \cP_{\Phi}$ and any 
$X\in H^{\Phi^{*}}$ due to Young's inequality. We shall call $\rho^{\Phi}$ the \textit{divergence risk measure w.r.t. $\Phi$}. 
\medskip

Ben-Tal and Teboulle (\cite{Ben-TalTeboulle1987}, 
\cite{Ben-TalTeboulle2007}) discovered another more convenient representation. 
It reads as follows (see \cite{BelomestnyKraetschmer2016}).
\begin{theorem}
\label{optimized certainty equivalent}
The divergence risk measure $\rho^{\Phi}$ w.r.t. $\Phi$ satisfies the following representation
\begin{eqnarray*}
\rho^{\Phi}(X) 
= 
\inf_{x\in\R}\ex\left[\Phi^{*}(X + x) - x\right]
\quad\mbox{for all}~X\in H^{\Phi^{*}}.
\end{eqnarray*}
\end{theorem}
The representation in Theorem \ref{optimized certainty equivalent} is also known as 
the \textit{optimized certainty equivalent w.r.t. $\Phi^{*}$}. As optimized certainty equivalent the divergence measure $\rho^{\Phi}$ may be seen directly to be nondecreasing w.r.t. the increasing convex order. Theorem \ref{optimized certainty equivalent} also shows that $\rho^{\Phi}$ is distribution-invariant. In particular, we may define the functional $\cR_{\rho^{\Phi}}$ associated with $\rho^{\Phi}$ on the set $\F_{\Phi^{*}}$ of all distribution functions of the random variables from $H^{\Phi^{*}}$. Note that $\OFP$ supports some random variable $U$ which is uniformly distributed on $]0,1[$ because $\OFP$ is assumed to be atomless. Then we obtain for any distribution function $F\in\F_{\Phi^{*}}$ with left-continuous quantile function $\Flinks$ 
\begin{equation}
\label{Definition Divergenz Risikomass}
\cR_{\rho^{\Phi}}(F) ~=~ \rho^{\Phi}\big(\Flinks(U)\big) ~=~ \inf_{x\in\R}\left(\int_{0}^{1}\eins_{]0,1[}(u)~\Phi^{*}\big(\Flinks(u) + x\big)~du - x\right).
\end{equation}
For ease of reference we shall use notation $M_{F}$ to denote the set of all $x\in\RRR$ which solve the minimization in definition (\ref{Definition Divergenz Risikomass}) of 
$\cR_{\rho^{\Phi}}(F)$. In view of Proposition \ref{function of certainty equivalents basic properties} from the Appendix, each set $M_{F}$ is a nonvoid compact interval.
\medskip

Throughout this section we focus on the following specialization of optimization problem (\ref{optimization risk averse})
\begin{equation}
\label{optimization general II}
\inf_{\theta\in\Theta}\cR_{\rho^{\Phi}}\big(F_{\theta}\big),
\end{equation}
where $F_{\theta}$ stands for the distribution function of $G(\theta,Z)$ for 
$\theta\in\Theta$. The set of minimizers of the problem (\ref{optimization general II}) will be denoted by $\argmin\cR_{\rho^{\Phi}}$.
\par
The SAA (\ref{SAA risk averse}) of (\ref{optimization general II}) reads as follows. 
\begin{equation}
\label{optimization approximativ}
\inf_{\theta\in\Theta}\cR_{\rho^{\Phi}}\big(\hat{F}_{n,\theta}\big) 
= \inf_{\theta\in\Theta}\inf_{x\in\RRR}\Big(\frac{1}{n}\sum_{i=1}^{n}\Phi^{*}\big( G(\theta, Z_{i}) +x\big) - x\Big)\quad(n\in\NNN).
\end{equation}
We shall strengthen condition (A2) to the following property. 
\smallskip

\noindent
\begin{itemize}
\item [(A2')] There exists some positive envelope $\xi$ of $\FFF^{\Theta}$ satisfying $\xi(Z_{1})\in\cH^{\Phi^{*}}$.
\end{itemize}
\smallskip

\noindent
Note that (A2') together with (A1) implies that $G(\theta,Z_{1})$ belongs to $\cH^{\Phi^{*}}$ for every $\theta\in\Theta$ so that the genuine optimization problem \eqref{optimization general II} is well-defined.
\smallskip
According to Theorem 4.6 in \cite{Kraetschmer2022}
$$
n^{a}~\Big[\inf\limits_{\theta\in\Theta}\cR_{\rho^{\Phi}}\big(\hat{F}_{n,\theta}\big) - \inf\limits_{\theta\in\Theta}\cR_{\rho^{\Phi}}\big(F_{\theta}\big)\Big]\underset{n\to\infty}{\to} 0\quad\mbox{in probability}\quad\mbox{for}~a\in ]0,1/2[,
$$ 
and the sequence
\begin{equation}
\label{sequence object2}
\Big(\sqrt{n}~\Big[\inf\limits_{\theta\in\Theta}\cR_{\rho^{\Phi}}\big(\hat{F}_{n,\theta}\big) - \inf\limits_{\theta\in\Theta}\cR_{\rho^{\Phi}}\big(F_{\theta}\big)\Big]\Big)_{n\in\NNN}
\end{equation}
is uniformly tight. The essential requirements are assumptions (A1), (A2') and the finiteness of $\overline{J}(\FFF^{\Theta}, \xi,1)$, where $\xi$ is from (A2'). In this section we want to derive asymptotic distributions of  the sequence \eqref{sequence object2}.
\smallskip

Representation \eqref{optimization approximativ} along with Theorem \ref{optimized certainty equivalent} suggests to apply Theorem \ref{asymptoticDistributionRiskNeutral} to the SAA of 
\begin{equation}
\label{auxiliary problem}
\inf_{(\theta,x)\in\Theta\times\RRR}\EEE\big[G_{\Phi}\big((\theta,x),Z_{1}\big)\big],
\end{equation}
where
\begin{equation}
\label{new objective}
G_{\Phi}: (\Theta\times\RRR)\times\RRR^{d}\rightarrow\RRR, \big((\theta,x),z\big)\mapsto \Phi^{*}\big(G(\theta,z) + x\big) - x.
\end{equation}
Unfortunately, the application is not immediate because the parameter space is not totally bounded w.r.t. the Euclidean metric on $\RRR^{m + 1}$. We already know that the solution set of the optimization problem \eqref{auxiliary problem} is compact under (A1), (A2') (see \cite[Lemma 5.8]{Kraetschmer2022}). Conditions (A1), (A2') also imply that the associated SAA problems have nonvoid compact solution sets. Unfortunately, they may depend on the realizations of the samples. In \cite{Kraetschmer2022} a kind of compactification was suggested which allows to restrict the parameter set of the random process $G_{\Phi}$ to suitable compact subsets. The idea is to show that with arbitrarily high probability we may find for large sample sizes events from $\mathcal{F}$ on which all solution sets of the SAA problems are contained in a common compact superset. The following result from \cite{Kraetschmer2022} gives a precise formulation of this idea. 
For preparation consider any mapping $\xi$ as in (A2') and let us introduce for $\delta > 0$ and $n\in\NNN$ 
$$
A_{n,\delta}^{\xi} := \Big\{\frac{1}{n}\sum_{j=1}^{n}\xi(Z_{j})\leq \EEE[\xi(Z_{1})] + \delta,~\frac{1}{n}\sum_{j=1}^{n}\Phi^{*}\big(\xi(Z_{j})\big)\leq  \EEE\big[\Phi^{*}\big(\xi(Z_{1})\big)\big] + \delta\Big\}.
$$
Note that $A_{n,\delta}^{\xi}$ belongs to $\cF$, and $\MP(A_{n,\delta}^{\xi})\to 1$ for $n\to\infty$ due to the law of large numbers. The following result has been shown in \cite{Kraetschmer2022} (Theorem 5.7 with Lemma 5.8).
\begin{proposition}
\label{stochastic equicontinuity}
Let (A1), (A2') be fulfilled. Then the set of solutions of problem \eqref{auxiliary problem} is nonvoid and compact, and there always exists a solution of \eqref{optimization approximativ} for any $\omega$. Furthermore with mapping $\xi$ from (A2'), for every $\delta > 0$ and $n\in\NNN$ there is some $k_{\delta}\in\NNN$ such that
\begin{align*} 
\Big\{(\theta,x)\in\Theta\times\RRR\mid \EEE\big[G_{\Phi}\big((\theta,x),Z_{1}\big)\big] = \inf_{\theta\in\Theta\atop x\in\RRR}\EEE\big[G_{\Phi}\big((\theta,x),Z_{1}\big)\big]\Big\}\subseteq \Theta\times [-k_{\delta},k_{\delta}],
\end{align*}
and
\begin{align*}
\Big\{(\theta,x)\in\Theta\times\RRR\mid \frac{1}{n}\sum_{j=1}^{n}G_{\Phi}\big((\theta,x),Z_{j}(\omega)\big) = \inf_{\theta\in\Theta\atop x\in\RRR}\sum_{j=1}^{n}G_{\Phi}\big((\theta,x),Z_{j}(\omega)\big)\Big\}\subseteq \Theta\times [-k_{\delta},k_{\delta}]
\end{align*}
for $\delta > 0$, $n\in\NNN$ and $\omega\in A_{n,\delta}^{\xi}$.
\end{proposition}
Based upon Proposition \ref{stochastic equicontinuity} it will turn out that it is already sufficient to apply Theorem \ref{asymptoticDistributionRiskNeutral} to the SAA corresponding to the 
function classes of the following type
\begin{equation}
\label{neue Funktionsklassen}
\FFF^{\Theta}_{\Phi,k} :=\big\{G_{\Phi}\big((\theta,x),\cdot\big)\mid (\theta,x)\in\Theta\times [-k,k]\big\}\quad(k\in\NNN).
\end{equation}
The finiteness of the terms $\overline{J}(\FFF^{\Theta}_{\Phi,k},C_{\FFF^{\Theta}_{\Phi,k}},1)$ is already guaranteed by finiteness of the terms $\overline{J}(\FFF^{\Theta},C_{\FFF^{\Theta}},1)$ associated with the genuine objective $G$. This will be the subject of the following result which has been proved in \cite{Kraetschmer2022} (Lemma 4.3).
\begin{lemma}
\label{relationship}
Let $\Phi^{*'}_{+}$ denote the right-sided derivative of $\Phi^{*}$. If $\xi$ is a square $\MP^{Z}$-integrable positive envelope of $\FFF^{\Theta}$, then for any $k\in\NNN$ the mapping
$$
C_{\FFF^{\Theta}_{\Phi,k}} := 2 \big[\Phi^{*'}_{+}\big(\xi + k\big) + 1] \sqrt{\xi^{2} + k^{2}}
$$
is a positive envelope of $\FFF^{\Theta}_{\Phi,k}$  satisfying $\overline{J}(\FFF^{\Theta}_{\Phi,k},C_{\FFF^{\Theta}_{\Phi,k}},1) < \infty$ if $\overline{J}(\FFF^{\Theta},\xi,1)$ is finite.
\end{lemma}
Next, if $G_{\Phi}\big((\theta,x),\cdot\big)$ is square $\MP^{Z}$-integrable for $(\theta,x)\in\Theta\times\RRR$, then we shall endow $\Theta\times \RRR$ with the semimetric $\overline{d}_{\Theta,\Phi}$ defined by 
$$
\overline{d}_{\Theta,\Phi}\big((\theta,x),(\vartheta,y)\big) := \sqrt{\vari\big[G_{\Phi}\big((\theta,x),Z_{1}\big) - G_{\Phi}\big((\vartheta,y),Z_{1}\big)\big]}.
$$
Now the application of Theorem \ref{asymptoticDistributionRiskNeutral} to the restricted optimization problems associated with the function classes $\FFF^{\Theta}_{\Phi,k}$ reads as follows.
\begin{proposition}
\label{restricted optimization}
Let (A1), (A2') be fulfilled. The Borel measurable mapping $\xi$ from (A2') is assumed to satisfy the property that the mapping $\xi_{k} := [\Phi^{*'}_{+}(\xi + k) + 1]\sqrt{\xi^{2} + k^{2}}$ is square $\MP^{Z}$-integrable for every $k\in\NNN$. 
If $\overline{J}(\FFF^{\Theta},\xi,1)$ is finite, then the following statements are true.
\begin{itemize}
\item [1)] $\inf\limits_{\theta\in\Theta\atop x\in [-k,k]}\frac{1}{n}\sum\limits_{j=1}^{n}G_{\Phi}\big((\theta,x),Z_{j}\big) - \inf\limits_{\theta\in\Theta\atop x\in [-k,k]}\EEE[G_{\Phi}\big((\theta,x),Z_{1}\big)]$ is a random variable on the probability space $\OFP$ for arbitrary $k,n\in\NNN$.
\item [2)] The sets 
$$
\mathcal{S}_{k} := \Big\{(\theta,x)\in\Theta\times [-k,k]\mid \EEE\big[G_{\Phi}\big((\theta,x),Z_{1}\big)\big] = \inf_{\theta\in\Theta\atop x\in [-k,k]}\EEE\big[G_{\Phi}\big((\theta,x),Z_{1}\big)\big]\Big\}~ (k\in\NNN)
$$
are nonvoid and compact.
\item [3)] For $k\in\NNN$ the semimetric $\overline{d}_{\Theta,\Phi}$ is totally bounded on $\Theta\times [-k,k]$, and there exists some centered Gaussian process $\mathfrak{G}^{k} = (\mathfrak{G}^{k}_{(\theta,x)})_{(\theta,x)\in\Theta\times [-k,k]}$ such that 
the sequence 
$$
\Big(\sqrt{n}\Big[\inf\limits_{\theta\in\Theta\atop x\in [-k,k]}\frac{1}{n}\sum\limits_{j=1}^{n}G_{\Phi}\big((\theta,x),Z_{j}\big) - \inf\limits_{\theta\in\Theta\atop x\in [-k,k]}\EEE[G_{\Phi}\big((\theta,x),Z_{1}\big)]\Big]\Big)_{n\in\NNN}
$$
converges weakly to 
$
\sup\limits_{(\theta,x)\in\mathcal{S}_{k}}\mathfrak{G}^{k}_{(\theta,x)}.
$
This Gaussian process has uniformly continuous paths w.r.t. $\overline{d}_{\Theta,\Phi}$ and satisfies 
$$
\EEE\big[\mathfrak{G}^{k}_{(\theta,x)}\cdot \mathfrak{G}^{k}_{(\vartheta,y)}\big] = 
\covi
\Big(G_{\Phi}\big((\theta,x),Z_{1}\big), G_{\Phi}\big((\vartheta,y),Z_{1})\Big)$$
for $\theta, \vartheta\in\Theta$ and $x, y\in [-k,k]$.
\end{itemize}
\end{proposition} 
\begin{proof}
Note that by (A1) the objective $G_{\Phi}$ is measurable w.r.t. the product $\sigma$-algebra $\cB(\Theta)\otimes\cR(\RRR^{d})$ and lower semicontinuous in the parameters $(\theta,x)$ because $\Phi^{*}$ is continuous and nonincreasing. Next, by Lemma \ref{relationship}, the mapping $\xi_{k}$ is a positive envelope of $\FFF^{\Theta}_{\Phi,k}$, which is also square $\MP^{Z}$-integrable by assumption. Hence $G_{\Phi}\big((\theta,x),\cdot\big)$ is square $\MP^{Z}$-integrable for $(\theta,x)\in\Theta\times [-k,k]$. Thus 
$\overline{d}_{\Theta,\Phi}$ is well-defined on $\Theta\times [-k,k]$ for $k\in\NNN$. 
\par
Now, the entire statement of Proposition \ref{restricted optimization} follows immediately from Proposition \ref{MessbarkeitLoesbarkeit}, and Theorem \ref{asymptoticDistributionRiskNeutral} along with Lemma \ref{relationship}. 
\end{proof}
Combining Proposition \ref{restricted optimization} with Proposition \ref{stochastic equicontinuity} we may derive our main result concerning first order asymptotics of the SAA \eqref{optimization approximativ}.
\begin{theorem}
\label{first order asymptotics divergence}
Let (A1), (A2') be fulfilled. The measurable mapping $\xi$ from (A2') is assumed to satisfy the property that the mapping $\xi_{k} := [\Phi^{*'}_{+}(\xi + k) + 1]\sqrt{\xi^{2} + k^{2}}$ is square $\MP^{Z}$-integrable for every $k\in\NNN$. 
If $\overline{J}(\FFF^{\Theta},\xi,1)$ is finite, then the following statements are valid.
\begin{itemize}
\item[1)]$\big(\inf\limits_{\theta\in\Theta}\cR_{\rho^{\Phi}}(\hF_{n,\theta}) - \inf\limits_{\theta\in\Theta}\cR_{\rho^{\Phi}}(F_{\theta})\big)_{n\in\NNN}$ is a sequence of random variables.
\item [2)] The sets $\argmin \cR_{\rho^{\Phi}}$ and $M_{F_{\theta}}$ ($\theta\in\Theta$) are nonvoid and compact.
\item [3)] For $k\in\NNN$ the semimetric $\overline{d}_{\Theta,\Phi}$ is totally bounded on $\Theta\times [-k,k]$, and there exists some centered Gaussian process $\mathfrak{G}^{\Phi} = (\mathfrak{G}^{\Phi}_{(\theta,x)})_{(\theta,x)\in\Theta\times\RRR}$ with 
$$
\EEE\big[\mathfrak{G}^{\Phi}_{(\theta,x)}\cdot \mathfrak{G}^{\Phi}_{(\vartheta,k)}\big] = 
\covi
\Big(G_{\Phi}\big((\theta,x),Z_{1}\big), G_{\Phi}\big((\vartheta,y),Z_{1})\Big)\quad\mbox{for}~\theta, \vartheta\in\Theta;~x, y\in \RRR
$$
such that 
the sequence $\big(\sqrt{n}[\inf\limits_{\theta\in\Theta}\cR_{\rho^{\Phi}}(\hF_{n,\theta}) - \inf\limits_{\theta\in\Theta}\cR_{\rho^{\Phi}}(F_{\theta})]\big)_{n\in\NNN}$ converges weakly to 
$$
\sup_{\theta\in\argmin \cR_{\rho^{\Phi}}}\sup_{x\in M_{F_{\theta}}}\mathfrak{G}^{\Phi}.
$$
In addition the paths of the Gaussian process are uniformly continuous on the set $\Theta\times [-k,k]$ w.r.t. $\overline{d}_{\Theta,\Phi}$ for $k\in\NNN$. Moreover, if the optimization problem (\ref{optimization general II}) has a unique solution $\theta^{*}$, and if $M_{F_{\theta^{*}}}$ has one element $x_{\theta^{*}}$ only, then the weak limit is some centered normally distributed random variable with variance $\vari\big(\Phi^{*}\big(G(\theta^{*},Z_{1}) + x^{*}\big)\big)$.
\end{itemize}
\end{theorem}

\begin{proof}
By Theorem \ref{optimized certainty equivalent} and representation \eqref{optimization approximativ} we have 
\begin{align}
&\label{startrepresentation}
\inf\limits_{\theta\in\Theta}\cR_{\rho^{\Phi}}(F_{\theta}) = \inf\limits_{(\theta,x)\in\Theta\times\RRR}\EEE[G_{\Phi}\big((\theta,x),Z_{1}\big)],\\
&\nonumber
\inf\limits_{\theta\in\Theta}\cR_{\rho^{\Phi}}(\hF_{n,\theta}) = \inf_{k\in\NNN}~\inf_{(\theta,x)\in\Theta\times [-k,k]}\frac{1}{n}\sum_{j=1}^{n}G_{\Phi}\big((\theta,x),Z_{j}\big)\quad\mbox{for}~n\in\NNN.
\end{align}
Then statement 1) may be concluded immediately from statement 1) of Proposition \ref{restricted optimization} because the optimal values of the optimization problemes \eqref{optimization general II} and \eqref{optimization approximativ} are always finite due to Proposition \ref{stochastic equicontinuity}. Concerning statement 2), we already know from Proposition \ref{function of certainty equivalents basic properties} in the Appendix that for any $\theta\in\Theta$, the set $M_{F_{\theta}}$ is nonvoid and compact. Furthermore, the set of minimizers of \eqref{auxiliary problem} is nonvoid and compact by Proposition \ref{stochastic equicontinuity}. Then $\theta$ belongs to $\argmin~\cR_{\rho^{\Phi}}$ if and only if $(\theta,x)$ is a solution of \eqref{auxiliary problem} for some $x\in\RRR$. Hence $\argmin~\cR_{\rho^{\Phi}}$ is nonvoid and compact so that statement 2) is verified. Hence it remains to show statement 3).
\smallskip

For this purpose let us select a sequence $(\frak{G}^{k})_{k\in\NNN}$ of centered Gaussian processes $\frak{G}^{k} = \big(\frak{G}^{k}_{(\theta,x)}\big)_{(\theta,x)\in\Theta\times [-k,k]}$ as in statement 3) of Proposition \ref{restricted optimization}.  With $\xi$ from 
(A2'), and defining for $n\in\NNN$ the set $A_{n,1}^{\xi}$ as in the text preceding Proposition \ref{stochastic equicontinuity}, we may find by Proposition \ref{stochastic equicontinuity} some $k_{1}\in\NNN$ such that the set of minimizers of $\EEE\big[G_{\Phi}(\cdot,Z_{1})\big]$ is contained in $\Theta\times [-k_{1},k_{1}]$,
and for $n\in\NNN, \omega\in A_{n,1}^{\xi}$
\begin{align*}
\inf_{(\theta,x)\in\Theta\times\RRR}\frac{1}{n}\sum_{j=1}^{N}G_{\Phi}\big((\theta,x),Z_{j}(\omega)\big) = 
\inf_{(\theta,x)\in\Theta\atop x\in [-k_{1},k_{1}]}\frac{1}{n}\sum_{j=1}^{N} G_{\Phi}\big((\theta,x),Z_{j}(\omega)\big).
\end{align*}
Next, by assumption, the sequences $\big(\xi(Z_{j})\big)_{j\in\NNN}$ and  $\big\{\Phi^{*}\big(\xi(Z_{j})\big)\big\}_{j\in\NNN}$ consist of independent integrable random variables which are identically distributed as $\xi(Z_{1})$ and $\Phi^{*}(\xi(Z_{1})\big)$ respectively. Then $\MP(\Omega\setminus A_{n,1}^{\xi}) \to 0$ by law of large numbers, and thus 
$$
\lim_{n\to\infty}\MP\Big(\Big\{\big|\sqrt{n}\big[\inf_{\theta\in\Theta\atop |x|\leq k_{1}}\frac{1}{n}\sum_{j=1}^{n}G_{\Phi}\big((\theta,x),Z_{j}\big)- \inf_{\theta\in\Theta\atop x\in\RRR}\frac{1}{n}\sum_{j=1}^{n}G_{\Phi}\big((\theta,x),Z_{j}\big)\big]\big| > \varepsilon\Big\}\Big) = 0
$$
for $\varepsilon > 0$. Note also that by choice of $k_{1}$ the set $\mathcal{S}_{\RRR}$ of solutions of \eqref{auxiliary problem} coincides with the set $\mathcal{S}_{k_{1}}$ of minimizers of $\EEE\big[G_{\Phi}(\cdot,Z_{1})\big]$ on $\Theta\times [-k_{1},k_{1}]$. Then in view of statement 3) of Proposition \ref{restricted optimization} along with \eqref{startrepresentation} and \eqref{optimization approximativ}
\begin{equation}
\label{Hauptkonvergenz}
\big(\sqrt{n}\big[ \inf_{\theta\in\Theta}\cR_{\rho^{\Phi}}(\hF_{n,\theta}) - \inf_{\theta\in\Theta}\cR_{\rho^{\Phi}}(F_{\theta})\big]\big)_{n\in\NNN}~\mbox{converges weakly to}~\inf_{(\theta,x)\in \mathcal{S}_{\RRR}}\frak{G}^{k_{1}}_{(\theta,x)}.
\end{equation}
Note also that the semimetric $\overline{d}_{\Theta,\Phi}$ is totally bounded on $\Theta\times [-k,k]$ for every $k\in\NNN$ by statement 3) of Proposition \ref{restricted optimization}. So we may find an isotone sequence $(\Gamma_{k})_{k\in\NNN}$ of at most countable sets such that $\Gamma_{k}\subseteq \Theta\times [-k,k]$ dense w.r.t. $\overline{d}_{\Theta,\Phi}$ for every $k\in\NNN$. Set $\Gamma := \bigcup_{k=1}^{\infty}\Gamma_{k}$.
\smallskip

By Kolmogorov's consistency theorem there exists some centered Gaussian process 
$\overline{\mathfrak{G}}^{\Phi} = \big(\overline{\mathfrak{G}}^{\Phi}_{(\theta,x)}\big)_{(\theta,x)\in\Gamma}$ on a probability space $\tOFP$ with covariances as in statement 3) of Theorem \ref{first order asymptotics divergence} (see e.g. \cite[Theorem 12.1.3]{Dudley2002}). Hence $\overline{\mathfrak{G}}^{\Phi}$ and $\mathfrak{G}^{k}$ have identical finite dimensional marginal distributions on $\Theta\times [-k,k]$ for $k\in\NNN$. This implies 
\begin{align*}
\EEE\big[\sup_{((\theta,x), (\vartheta,y))\in\cU^{k}_{\delta}}|\overline{\mathfrak{G}}^{\Phi}_{(\theta,x)} - \overline{\mathfrak{G}}^{\Phi}_{(\vartheta,y)}|\big]
=
\EEE\big[\sup_{((\theta,x), (\vartheta,y))\in\cU^{k}_{\delta}}|\mathfrak{G}^{k}_{(\theta,x)} - \mathfrak{G}^{k}_{(\vartheta,y)}|\big] \quad\mbox{for}~\delta > 0;~ k\in\NNN,
\end{align*}
where $\cU^{k}_{\delta} := \big\{\big((\theta,x),(\vartheta,y)\big)\in \Gamma_{k}\times\Gamma_{k}\mid \overline{d}_{\Theta,\Phi}\big((\theta,x),(\vartheta,y)\big) < \delta\big\}$. Since each $\mathfrak{G}^{k}$ has $\overline{d}_{\Theta,\Phi}$-uniformly continuous paths, we may invoke Theorem 4 from \cite{KuelbsEtAl2013} to conclude
\begin{align*}
\lim_{\delta\searrow 0}\EEE\big[\sup_{((\theta,x), (\vartheta,y))\in\cU^{k}_{\delta}}|\overline{\mathfrak{G}}^{\Phi}_{(\theta,x)} - \overline{\mathfrak{G}}^{\Phi}_{(\vartheta,y)}|\big]
=
0\quad\mbox{for}~ k\in\NNN.
\end{align*}
Then by Proposition \ref{extension result} in Appendix \ref{AppendixB} there exists some 
version $\frak{G}^{\Phi} = \big(\frak{G}^{\Phi}_{(\theta,x)}\big)_{(\theta,x)\in\Theta\times\RRR}$ of $\overline{\mathfrak{G}}^{\Phi}$ which has $\overline{d}_{\Theta,\Phi}$-uniformly continuous paths on $\Theta\times [-k,k]$ for $k\in\NNN$. In particular $\frak{G}^{\Phi}$ is a centered Gaussian process with covariances as in statement 3) of Theorem \ref{first order asymptotics divergence}. This also means that $\frak{G}^{\Phi}$ and $\frak{G}^{k_{1}}$ have identical finite-dimensional distributions on $\Theta\times [-k_{1},k_{1}]$, implying that 
$\inf\limits_{(\theta,x)\in\widetilde{\Gamma}}\frak{G}^{\Phi}$ and $\inf\limits_{(\theta,x)\in\widetilde{\Gamma}}\frak{G}^{k_{1}}$ are identically distributed for any nonvoid at most countable subset $\widetilde{\Gamma}\subseteq\Theta\times [-k_{1},k_{1}]$. Note that $\overline{d}_{\Theta,\Phi}$ is totally bounded on $\mathcal{S}_{\RRR}$, and thus 
separable. Recall also that 
the paths of the processes $\frak{G}^{\Phi}$ and $\frak{G}^{k_{1}}$ are uniformly continuous on $\Theta\times [-k_{1},k_{1}]$ w.r.t. $\overline{d}_{\Theta,\Phi}$. Therefore we may verify that the mappings $\inf\limits_{(\theta,x)\in\mathcal{S}_{\RRR}}\frak{G}^{\Phi}_{(\theta,x)}$ and 
$\inf\limits_{(\theta,x)\in\mathcal{S}_{\RRR}}\frak{G}^{k_{1}}_{(\theta,x)}$ are identically distributed random variables. Then in view of \eqref{Hauptkonvergenz} 
\begin{equation}
\label{Hauptkonvergenz2}
\big(\sqrt{n}\big[ \inf_{\theta\in\Theta}\cR_{\rho^{\Phi}}(\hF_{n,\theta}) - \inf_{\theta\in\Theta}\cR_{\rho^{\Phi}}(F_{\theta})\big]\big)_{n\in\NNN}~\mbox{converges weakly to}~\inf_{(\theta,x)\in \mathcal{S}_{\RRR}}\frak{G}^{\Phi}_{(\theta,x)}.
\end{equation}
Moreover, we already know from statement 2) that the sets $\argmin\cR_{\rho^{\Phi}}$ and $ M_{F_{\theta}}$ $(\theta\in\Theta)$ are nonvoid. We may also observe that $\argmin\cR_{\rho^{\Phi}} = \textrm{Pr}(\mathcal{S}_{\RRR})$ holds, where $\textrm{Pr}$ denotes the standard projection from $\Theta\times\RRR$ onto $\Theta$. 
Therefore
\begin{equation}
\label{andere Darstellung}
M_{F_{\theta}}\subseteq [-k_{1},k_{1}]~\mbox{for}~\theta\in\argmin R_{\rho^{\Phi}},~
\inf_{(\theta,x)\in \mathcal{S}_{\RRR}}\frak{G}^{\Phi}_{(\theta,x)} = \inf_{\theta\in \argmin R_{\rho^{\Phi}}}\inf_{x\in M_{F_{\theta}}}\frak{G}^{\Phi}_{(\theta,x)}.
\end{equation}
Combining \eqref{Hauptkonvergenz2} and \eqref{andere Darstellung} with the above mentioned properties of $\frak{G}^{\Phi}$, the first part of statement 3) follows immediately. The remaing part is an obvious consequence of the first one. The proof is complete.
\end{proof}
\begin{remark}
If $\Phi(0) = 0$, and if $\Phi^{*}$ is strictly convex on $]0,\infty[$, then $M_{F_{\theta}}$ is a singleton for every $\theta\in\Theta$ due to Proposition \ref{unique minimizer} in the Appendix below. In this situation we may obtain asymptotic normality in the statement 3) of Theorem \ref{first order asymptotics divergence} if the genuine optimization problem (\ref{optimization general II}) has a unique solution $\theta^{*}$.
\end{remark}
\begin{remark}
In case that $G$ either satisfies condition (H) or has representation (PH) Example \ref{Hoelder-Bedingung auxiliary} or Example \ref{startingpoint} respectively show how to find proper positive envelopes $\xi$ of $\FFF^{\Theta}$ meeting all requirements of Theorem \ref{first order asymptotics divergence}.
\end{remark}
Next let us illustrate the assumptions of Theorem \ref{first order asymptotics divergence} by the example of the so called \textit{Average Value at Risk} also known as \textit{Expected Shortfall}.
\begin{example}
\label{AVaR}
Let $\Phi$ be defined by $\Phi_{\alpha}(x) := 0$ for $x\leq 1/(1-\alpha)$ for some $\alpha\in ]0,1[$, and $\Phi(x) := \infty$ if $x > 1/(1-\alpha)$. Then $\Phi^{*}_{\alpha}(y) = y^{+}/(1-\alpha)$ for $y\in\R$. In particular $H^{\Phi^{*}}$ coincides with $L^{1}$, and we may recognize $\cR_{\rho^{\Phi}}$ as the so called \textit{Average Value at Risk} w.r.t. $\alpha$ (e.g. \cite{FoellmerSchied2011}, \cite{ShapiroEtAl}), i.e.
$$
\cR_{\rho^{\Phi}}(F) = \frac{1}{1-\alpha}~\int_{\Flinks(\alpha)}^{1}\eins_{]0,1[}(u)~\Flinks(u)~du = \inf_{x\in\R}\left(\int_{0}^{1}\eins_{]0,1[}(u)~\frac{(\Flinks(u) + x)^{+}}{1-\alpha}~du - x\right)
$$
(see e.g. \cite{KainaRueschendorf2007}). It may be verified easily that $M_{F_{0}} = [F^{\leftarrow}_{0}(\alpha),F^{\rightarrow}_{0}(\alpha)]$, where $F_{0}^{\rightarrow}$ denotes the right-continuous quantile function of $F_{0}\in \FFF_{\Phi^{*}}$.
\begin{itemize}
\item [1)] Let $\xi$ be a square $\MP^{Z}$-integrable positive envelope of $\FFF^{\Theta}$. Then $\xi$ satisfies (A2'), and the mapping $\xi_{k} := [\Phi^{*'}_{\alpha +}(\xi + k) + 1]\sqrt{\xi^{2} + k^{2}}$ is square $\MP^{Z}$-integrable for every $k\in\NNN$. In particular Theorem \ref{first order asymptotics divergence} carries over if (A1) is satisfied, and if $\overline{J}(\FFF^{\Theta},\xi,1)$ is finite.
\item [2)] If condition (H) is satisfied, and if $G(\overline{\theta},\cdot)$ is square $\MP^{Z}$-integrable for some $\overline{\theta}\in\Theta$, then we may find by Example \ref{Hoelder-Bedingung auxiliary} some square $\MP^{Z}$-integrable positive envelope of $\FFF^{\Theta}$ with finite $\overline{J}(\FFF^{\Theta},\xi,1)$. Hence in view of statement 1) Theorem \ref{first order asymptotics divergence} may be always applied under (A1).
\item [3)] In case that $G$ has representation (PH) being lower semicontinuous in $\theta$, there exists by Example \ref{startingpoint} some square $\MP^{Z}$-integrable positive envelope of $\FFF^{\Theta}$ with finite $\overline{J}(\FFF^{\Theta},\xi,1)$. Furthermore 
(A1) is automatically fulfilled. Therefore by statement 1) Theorem \ref{first order asymptotics divergence} may be applied.
\item [4] In \cite{GuiguesKraetschmerShapiro2018} the asymptotic result of Theorem \ref{first order asymptotics divergence} was obtained under (H) with $\beta = 1$. In addition convexity in $\theta$ was imposed on the goal function $G$ (see \cite[Theorem 2] {GuiguesKraetschmerShapiro2018}). However the investigations there are extended to optimization problems 
$$
\inf_{\theta\in\Theta}\sup_{w\in\frak{W}}\Big(w_{0}\EEE\big[G(\theta,Z_{1})\big] + \sum_{i=1}^{r}w_{i} \rho_{\Phi_{\alpha_{i}}}\big(G(\theta,Z_{1})\big)\Big)
$$
for fixed $\alpha_{1},\ldots,\alpha_{r}\in ]0,1[$. Here $\frak{W}$ denotes a nonvoid subset of $\Delta_{r+1}$ which consists of all $w\in [0,\infty[^{r+1}$ satisfying $w_{0}+\dots+w_{r} = 1$.
\end{itemize}
\end{example}
\begin{remark}
\label{confidence interval III}
 If the optimization problem (\ref{optimization general II}) has a unique solution $\theta^{*}$, and if $M_{F_{\theta^{*}}}$ has one element $x_{\theta^{*}}$ only, then with positive upper estimate $L\geq \sqrt{\vari\big(\Phi^{*}\big(G(\theta^{*},Z_{1}) + x^{*}\big)\big)}$ we may construct by Theorem \ref{first order asymptotics divergence} asymptotic confidence intervals on the optimal value of (\ref{optimization general II}) in the same way as described in Remark \ref{confidence interval I}. Alternatively, nonasymptotic confidence intervals may be built on upper estimates for the deviation probabilities 
$$
\MP\Big(\Big\{\big|\inf_{\theta\in\Theta}\cR_{\rho^{\Phi}}\big(\hat{F}_{n,\theta}\big) - \inf_{\theta\in\Theta}\cR_{\rho^{\Phi}}\big(F_{\theta}\big)\big|\geq\varepsilon\Big\}\Big)\quad(n\in\NNN,\varepsilon > 0)
$$
from Theorem 4.6 in \cite{Kraetschmer2022}. In view of Example \ref{startingpoint} both approaches may be applied to objectives satisfying representation (PH), e.g. in two stage mixed-integer programs.
\par
In the special case of $\rho$ being the Average Value at Risk it would be interesting to compare the confidence intervals corresponding to both methods with those confidence which are obtained with stochastic mirror descent as in Subsection 5.2.2 of \cite{GuiguesKraetschmerShapiro2018}. However, for this purpose we have to impose further regularity conditions on the objective at least convexity and continuity in the parameters.
 \end{remark}
\section{The basic technical result}
\label{basic result}
In Sections \ref{FOA absolute semideviations}, \ref{generalized divergence risk measures} the main convergences results, namely Theorems \ref{first asymptotics absolute semideviations}, \ref{first order asymptotics divergence}, are derived as applications of Theorem \ref{asymptoticDistributionRiskNeutral}. Roughly speaking our verification of Theorem \ref{asymptoticDistributionRiskNeutral} combines a convergence theorem for empirical processes with the functional delta method for optimal values. This section provides a suitable convergence result for empirical processes adapted to the situation of the paper.
\smallskip

We shall use the following notations. For some nonvoid set $\mathbb{T}$ we shall denote by $l^{\infty}(\mathbb{T})$ the space of all bounded real-valued mappings on $\mathbb{T}$. It will be endowed with the supremum norm $\|\cdot\|_{\mathbb{T},\infty}$ and the induced Borel $\sigma$-algebra.
\medskip

Let us introduce the random processes 
$$
Y_{n}: \Theta\times\Omega\rightarrow\RRR,~(\theta,\omega)\mapsto \frac{1}{n}\sum_{j=1}^{n}\Big(G\big(\theta,Z_{j}(\omega)\big) - \EEE[G(\theta,Z_{1})]\Big)\quad(n\in\NNN). 
$$
If the objective $G$ satisfies the properties (A1) and (A2), then $Y_{n}(\cdot,\omega)$ belongs to $l^{\infty}(\Theta)$ for every $\omega\in\Omega$. 
\smallskip

Under (A1), (A2) the mapping $\theta\mapsto \EEE[G(\theta,Z_{1})]$ on $\Theta$ is lower semicontinuous (see Lemma \ref{basic observations} below). Hence each process $Y_{n}$ is $\cB(\Theta)\otimes\cB(\RRR^{d})$-measurable by (A1), $\Theta$ is a Polish space, and $\OFP$ is complete. Then $\sup_{\theta\in\Theta}|Y_{n}(\theta,\cdot)| = \|Y_{n}\|_{\Theta,\infty}$ is a random variable on $\OFP$ (see \cite[Lemma 1.7.5]{vanderVaartWellner1996}). We shall study the convergence of the sequence built upon these random variables.
\smallskip

Turning over to the issue of asymptotic distributions of the random processes $Y_{n}$, we are faced with the inconvenience that they might not be Borel random elements of $l^{\infty}(\Theta)$. Hence in general we may not apply weak convergence to the random processes $Y_{n}$. Fortunately, for our purposes it is sufficient to look instead when the sequence $(\sqrt{n}~Y_{n})_{n\in\NNN}$ of empirical processes $\sqrt{n}~Y_{n}$ converges in law (in the sense of Hoffmann-Jorgensen) to some tight Borel random element of $l^{\infty}(\Theta)$. 
Recall that for a sequence $\big((\overline{\Omega}_{n},\overline{\cF}_{n},\overline{\MP}_{n})\big)_{n\in\NNN}$ of probability spaces and a metric space $(\mathcal{D},d_{\mathcal{D}})$, a sequence $(\overline{W}_{n})_{n\in\NNN}$ of mappings $\overline{W}_{n}:\overline{\Omega}_{n}\rightarrow\mathcal{D}$ is called to \textit{converge in law} (in the sense of Hoffmann-Jorgensen) to a Borel random element $\overline{W}$ of $\mathcal{D}$ if the convergence $\EEE_{n}^{*}[f(\overline{W}_{n})]\to \EEE[f(\overline{W})]$ holds for every bounded continuous $f:\mathcal{D}\rightarrow\RRR$. Here $\EEE_{n}^{*}$ is used to denote the outer expectation w.r.t. to $\overline{\MP}_{n}$. For introduction and further studies of this kind of convergence we recommend \cite{vanderVaartWellner1996}, where, however, it is called weak convergence. Note that the mappings are not required to be Borel random elements of $\mathcal{D}$. Thus convergence in law differs from the usual weak convergence of Borel random elements. We decided to emphasize this difference by avoiding the term weak convergence. Obviously both concepts coincide if the involved mappings are Borel random elements of $\mathcal{D}$.
\medskip

Let us remind the semimetric $\overline{d}_{\Theta}$ on $\Theta$, defined just before Theorem \ref{asymptoticDistributionRiskNeutral}. Our basic technical result is the following criterion which guarantees almost sure convergence of $\big(\|Y_{n}(\theta,\cdot)\|_{\Theta,\infty}\big)_{n\in\NNN}$, and convergence in law of $(\sqrt{n}~Y_{n})_{n\in\NNN}$ to some tight centered Gaussian random element of $l^{\infty}(\Theta)$.
\begin{theorem}
\label{main technical result}
Let (A1), (A2) be fulfilled, where the mapping $\xi$ from (A2) is square $\MP^{Z}$-integrable.  
Using notation \eqref{Entropie-Integral I}, if $\overline{J}(\FFF^{\Theta},\xi,1)$ is finite, then the following statements are valid. 
\begin{itemize}
\item [1)] $\lim\limits_{n\to\infty}\|Y_{n}\|_{\Theta,\infty} = 0$ $\MP$-a.s..
\item [2)] 
$\overline{d}_{\Theta}$ is totally bounded, and there exists some tight random element $\mathfrak{G}$ of $l^{\infty}(\Theta)$ such that 
the sequence $(\sqrt{n}~Y_{n})_{n\in\NNN}$ converges in law to $\mathfrak{G}$. 
This tight random element is a centered Gaussian process $\mathfrak{G} = (\mathfrak{G}_{\theta})_{\theta\in\Theta}$ which has uniformly continuous paths w.r.t. $\overline{d}_{\Theta}$, satisfying in addition 
$$
\EEE\big[\mathfrak{G}_{\theta}\cdot \mathfrak{G}_{\vartheta}\big] = 
\covi
\big(G(\theta,Z_{1}), G(\vartheta,Z_{1})\big)\quad\mbox{for}~\theta, \vartheta\in\Theta.
$$
\end{itemize}
\end{theorem}
Theorem \ref{main technical result} has the following corollary which will turn out to be useful in the context of the SAA method under absolute semideviation. 
\begin{corollary}
\label{technical corollary}
Define for any nonvoid compact interval $\mathcal{I}\subseteq\RRR$ the real valued mapping $\overline{G}_{\mathcal{I}}$ on $(\Theta\times\mathcal{I})\times\RRR^{d}$ via $\overline{G}_{\mathcal{I}}\big((\theta,t),z\big) := \big(G(\theta,z)- t)^{+}$. Then under the assumptions of Theorem \ref{main technical result} the mappings 
$$
\overline{X}^{\mathcal{I}}_{n}: (\Theta\times\mathcal{I})\times\Omega\rightarrow\RRR,~\big((\theta,t), \omega\big)\mapsto 
\frac{1}{n}\sum_{j=1}^{n}\overline{G}_{\mathcal{I}}\big((\theta,t),Z_{j}(\omega)\big) -\EEE\big[\overline{G}_{\mathcal{I}}\big((\theta,t),Z_{1}\big)\big]\quad(n\in\NNN)
$$
satisfy 
$\overline{X}^{\mathcal{I}}_{n}(\cdot,\omega)\in l^{\infty}(\Theta\times\mathcal{I})$ for every $\omega\in\Omega$, and the sequence $(\sqrt{n}\overline{X}^{\mathcal{I}}_{n})_{n\in\NNN}$ converges in law to some tight centered Gaussian random element of $l^{\infty}(\Theta\times\mathcal{I})$.
\end{corollary}
\begin{proof}
We may observe 
\begin{align*}
\big|\overline{G}_{\mathcal{I}}\big((\theta,t),z\big) - \overline{G}_{\mathcal{I}}\big((\vartheta,s),z\big)\big|^{2}
\leq 
4\big(|G(\theta,z) - G(\vartheta,z)|^{2} + |t - s|^{2}\big)
\end{align*}
for $\theta,\vartheta\in\Theta; t, s\in\RRR$. Property (A2) is assumed to hold for some square $\MP^{Z}$-integrable mapping $\xi$. Then $\xi + |\inf\mathcal{I}| + |\sup\mathcal{I}|$ is a square $\MP^{Z}$-integrable positive envelope of $\big\{\overline{G}_{\mathcal{I}}\big((\theta,t),\cdot\big)\mid (\theta,t)\in \Theta\times\mathcal{I} \big\}$. Now, the statement of Corollary \ref{technical corollary} may be derived easily from Theorem \ref{main technical result}, using Corollary 2.10.13 from \cite{vanderVaartWellner1996}.
\end{proof}
\section{Proofs}
\label{Beweise}
Let us introduce the sequence 
$\big(X_{n}\big)_{n\in\NNN}$ of random processes  
$$
X_{n}:\Omega\times\Theta\rightarrow\RRR,~X_{n}(\omega,\theta) := \frac{1}{n}~\sum_{i=1}^{n}G\big(\theta,Z_{i}(\omega)\big) \quad(n\in\NNN),
$$
and under (A2) the mapping
$
\psi:\theta\rightarrow\RRR,~\theta\mapsto \EEE[G(\theta,Z_{1})].
$
First of all we want to fix lower semicontinuous of the mapping $\psi$.
\begin{lemma}
\label{basic observations}
Under (A1), (A2) the mapping $\psi$ is lower semicontinuous. It is even continuous if in addition property (A3) holds. 
\end{lemma}
\begin{proof}
With $\xi$ from (A2) the mapping $G\big(\cdot,Z_{1}(\omega)\big) + \xi\big(Z_{1}(\omega)\big)$ is nonnegative and lower semicontinuous for $\omega\in\Omega$ due to (A2) along with (A1).
Then an application of Fatou's Lemma shows that $\psi + \EEE[\xi(Z_{1})]$ is lower semicontinuous. This implies lower semicontinuity of $\psi$. If in addition (A3) is fulfilled, then continuity of $\psi$ follows directly from Vitalis' theorem (see \cite[Proposition 21.4]{Bauer2001}). This completes the proof.
\end{proof}
\subsection{Proof of Proposition \ref{MessbarkeitLoesbarkeit}}
\label{Beweis von Proposition MessbarkeitLoesbarkeit}
%
The mapping $\psi$ is lower semicontinuous due to Lemma \ref{basic observations}. Then statement 2) follows immediately due to compactness of $\Theta$.
\smallskip

Concerning statement 1) let $n\in\NNN$. The mapping $X_{n}(\omega,\cdot)$ is bounded from below on $\Theta$ for every $\omega\in\Omega$ by (A2). Furthermore 
$$
\big\{\omega\in\Omega\mid \inf_{\theta\in\Theta}X_{n}(\omega,\theta) < t\big\} = \textrm{Pr}_{\Omega}\big(\big\{(\omega,\theta)\in\Omega\times\Theta\mid X_{n}(\omega,\theta) < t\big\}\big)\quad\mbox{for}~t\in\RRR,
$$
where $\textrm{Pr}_{\Omega}$ denotes the standard projection from $\Omega\times\Theta$ onto $\Omega$. By the assumption (A1) the set $\big\{(\theta,\omega)\in\Omega\times\Theta\mid X_{n}(\omega,\theta) < t\big\}$ belongs to $\cF\otimes\cB(\Theta)$ for every $t\in\RRR$. Since $\Theta$ is a Polish space, and since $\OFP$ is complete, we may conclude that the set $\big\{\omega\in\Omega\mid \inf_{\theta\in\Theta}X_{n}(\cdot,\theta) < t\big\}$ is a member of $\cF$ (see \cite[Proposition 8.4.4]{Cohn1980}). In particular, $\inf_{\theta\in\Theta}X_{n}(\cdot,\theta)$ is a random variable on $\OFP$ which completes the proof.
%
\hfill$\Box$

\subsection{Proof of Theorem \ref{main technical result} and Theorem  \ref{asymptoticDistributionRiskNeutral}}
\label{Beweis von Hauptresultat}
Let $(Y_{n})_{n\in\NNN}$ be the sequence of stochastic processes introduced in Section \ref{basic result}. In order to show the desired convergences of $\big(\|Y_{n}\|_{\Theta,\infty}\big)_{n\in\NNN}$ and $(\sqrt{n}~Y_{n})_{n\in\NNN}$ in $l^{\infty}(\Theta)$ we shall invoke results from empirical process theory. We have to circumvent some subtleties of measurability, reminding the notion of $\MP^{Z}$-measurable classes. A class $\FFF$ of Borel measurable mappings from $\RRR^{d}$ into $\RRR$ is called a $\MP^{Z}$\textit{-measurable class} if for $n\in\NNN$ and $a_{1},\ldots,a_{n}\in\RRR^{n}$, and every $h\in\FFF$ the mapping 
$$
\RRR^{nd}\rightarrow\RRR,~(z_{1},\ldots,z_{n})\mapsto \sup_{h\in\FFF}\big|\sum_{j=1}^{n}a_{i} h(z_{j})\big|
$$
is well-defined and measurable on the completion of the $n$-times product probability space $\big(\RRR^{nd},\cB(\RRR^{nd}),(\MP^{Z})^{n}\big)$ of $\big(\RRR^{d},\cB(\RRR^{d}),\MP^{Z}\big)$ 
(see \cite[Definition 2.3.3]{vanderVaartWellner1996}).
\smallskip

Define for a nonvoid nonvoid $\Gamma\subseteq\Theta\times\Theta$ the function class $\FFF_{\Gamma}^{\Theta}$ consisting of all $G(\theta,\cdot) - G(\vartheta,\cdot)$ with $\theta, \vartheta\in\Gamma$, and set 
$\FFF^{\Theta,2}_{\Gamma} := \big\{|G(\theta,\cdot) - G(\vartheta,\cdot)|^{2}\mid \theta, \vartheta\in\Gamma\big\}$. If $\Gamma$ is a Borel subset of $\Theta\times\Theta$ these classes are already $\MP^{Z}$-measurable classes which is subject of the following result.
\begin{lemma}
\label{Fk measurable class}
Let $\Gamma$ be some nonvoid Borel subset of $\Theta\times\Theta$. If (A1) and (A2) are satisfied, then $\FFF^{\Theta}$, $\FFF^{\Theta}_{\Gamma}$ and $\FFF^{\Theta,2}_{\Gamma}$ are $\MP^{Z}$-measurable classes.
\end{lemma}
\begin{proof}
Under (A1) and (A2), for $n\in\NNN$, $(a_{1},\ldots,a_{n})\in\RRR^{n}$ the processes
\begin{align*}
&
\big((z_{1},\ldots,z_{n}), (\theta,\vartheta)\big)\mapsto \sum_{j=1}^{n}a_{j}~\big[G(\theta,z_{j}) - G(\theta,z_{j})\big]\\
&
\big((z_{1},\ldots,z_{n}), (\theta,\vartheta)\big)\mapsto \sum_{j=1}^{n}a_{j}~\big[G(\theta,z_{j}) - G(\theta,z_{j})\big]^{2}
\end{align*}
on $\RRR^{d n}\times \Gamma$ are measurable w.r.t. the product 
$\cB(\RRR^{d n})\otimes\cB(\Gamma)$ of the Borel $\sigma$-algebra $\cB(\RRR^{d n})$  on $\RRR^{dn}$ and the Borel $\sigma$-algebra $\cB(\Gamma)$ on $\Gamma$ with 
$$
\sup_{(\theta,\vartheta)\in\Gamma}\big|\sum_{j=1}^{n}a_{j}~[G(\theta,z_{j}) - G(\theta,z_{j})]\big|~\vee~ \sup_{(\theta,\vartheta)\in\Gamma}\big|\sum_{j=1}^{n}a_{j}~[G(\theta,z_{j}) - G(\theta,z_{j})]^{2}\big| < \infty
$$
for $(z_{1},\ldots,z_{n})\in\RRR^{dn}$. Since $\Gamma$ is a Borel subset of a Polish space, the mappings
\begin{align*}
&
\RRR^{dn}\rightarrow\RRR,~(z_{1},\ldots,z_{n})\mapsto \sup_{(\theta,\vartheta)\in\Gamma}\big|\sum_{j=1}^{n}a_{j}~[G(\theta,z_{j}) - G(\theta,z_{j})]\big|\\
&
\RRR^{dn}\rightarrow\RRR,~(z_{1},\ldots,z_{n})\mapsto\sup_{(\theta,\vartheta)\in\Gamma}\big|\sum_{j=1}^{n}a_{j}~[G(\theta,z_{j}) - G(\theta,z_{j})]^{2}\big|
\end{align*}
are random variables on the completion of the probability space $\big(\RRR^{dn},\cB(\RRR^{d n}),(\MP^{Z})^{n}\big)$ for $n\in\NNN$ and $(a_{1},\ldots,a_{n})\in\RRR^{n}$ (see \cite[Example 1.7.5]{vanderVaartWellner1996}). 
\medskip

In exactly the same way we may also verify $\FFF^{\Theta}$ as a $\MP^{Z}$-measurable 
class. This completes the proof.
\end{proof}
Now, we are ready to show Theorem \ref{main technical result}.
\medskip

\noindent
\textbf{Proof of Theorem \ref{main technical result}:}\\[0.1cm]
Let $\tOFP = \big((\RRR^{d})^{\NNN},\cB(\RRR^{d})^{\otimes\NNN},(\MP^{Z})^{\NNN}\big)$ be the countable product probability space of $\big(\RRR^{d},\cB(\RRR^{d}), \MP^{Z}\big)$. Furthermore for $j\in\NNN$, the mapping $\pi_{j}: (\RRR^{d})^{\NNN}\rightarrow\RRR^{d}$ is defined by $\pi_{j}\big((z_{i})_{i\in\NNN}\big) = z_{j}$.
\smallskip

For any real valued mapping $f$ on $\overline{\Omega}$ the set $\mathcal{E}_{f}$ of all 
mappings $\overline{f}: \overline{\Omega}\rightarrow\RRR\cup\{\infty\}$ satisfying $\overline{f}\geq f$ pointwise and $\{\overline{f} > x\}\in\overline{\cF}$ for $x\in\RRR$ is nonvoid. Moreover, there exists some $f^{*}\in \mathcal{E}_{f}$ such that $f^{*}\leq\overline{f}$ $\overline{\MP}$-a.s. for every $\overline{f}\in \mathcal{E}_{f}$ (see \cite[Lemma 1.2.1]{vanderVaartWellner1996}). It is also known that $(\lambda f)^{*} = \lambda f^{*}$ holds for any $\lambda > 0$ (see \cite[Lemma 1.2.2]{vanderVaartWellner1996}).
\smallskip

Under assumption (A2) the mappings $\big(\sup_{\theta\in\Theta}|\sum_{j=1}^{n}G(\theta,\pi_{j})/n - \EEE_{\overline{\MP}}[G(\theta,\pi_{1})]|\big)^{*}$ have finite values. Then by Markov's inequality
\begin{align*}
&
\overline{\MP}\Big(\Big\{\big(\sup_{\theta\in\Theta}|\frac{1}{n}\sum_{j=1}^{n}G(\theta,\pi_{j}) - \EEE_{\overline{\MP}}[G(\theta,\pi_{1})]|\big)^{*} > \varepsilon\Big\}\Big)
\\
& 
\leq 
\EEE_{\overline{\MP}}\Big[\big(\sup_{\theta\in\Theta}|\frac{1}{\sqrt{n}}\sum_{j=1}^{n}G(\theta,\pi_{j}) - \sqrt{n}\EEE_{\overline{\MP}}[G(\theta,\pi_{1})]|\big)^{*}\Big]/(\sqrt{n}\varepsilon)
\quad\mbox{for}~n\in\NNN.
\end{align*}
Since $\FFF^{\Theta}$ is a $\MP^{Z}$-measurable class by Lemma \ref{Fk measurable class}, and since the positive envelope $\xi$ is assumed to be square $\MP^{Z}$-integrable, we may apply Theorem 2.14.1 from \cite{vanderVaartWellner1996} to find a constant $M > 0$ such that 
$$
\EEE_{\overline{\MP}}\Big[\big(\sup_{\theta\in\Theta}|\frac{1}{\sqrt{n}}\sum_{j=1}^{n}G(\theta,\pi_{j}) - \sqrt{n}\EEE_{\overline{\MP}}[G(\theta,\pi_{1})]|\big)^{*}\Big]\leq M \big[1 + \overline{J}(\FFF^{\Theta},\xi,1)\big]~\|\xi\|_{\MP^{Z},2} 
\quad\mbox{for}~n\in\NNN.
$$
In particular, by finiteness of $\overline{J}(\FFF^{\Theta},\xi,1)$, we end up with 
$$
\lim_{n\to\infty}\overline{\MP}\Big(\Big\{\big(\sup_{\theta\in\Theta}|\frac{1}{n}\sum_{j=1}^{n}G(\theta,\pi_{j}) - \EEE_{\overline{\MP}}[G(\theta,\pi_{1})]|\big)^{*} > \varepsilon\Big\}\Big) = 0.
$$
Then in view of Corollary 3.7.9 from \cite{GineNickl2016} statement 1) follows immediately.
\medskip

Concerning statement 2) let for $\delta > 0$ the set $\Gamma_{\delta}$ consist of all $(\theta,\vartheta)\in\Theta\times\Theta$ such that $\|G(\theta,\cdot) - G(\vartheta,\cdot)\|_{\MP^{Z},2} < \delta$. Note that all $G(\theta,\cdot)$ are square $\MP^{Z}$-integrable because the positive envelope $\xi$ of $\FFF^{\Theta}$ is assumed to be square $\MP^{Z}$-integrable. The mapping $\big((\theta,\vartheta),z\big)\mapsto |G(\theta,z) - G(\vartheta,z)|^{2}$ on $(\Theta\times\Theta)\times\RRR^{d}$ is measurable w.r.t. the product $\sigma$-algebra $\cB(\Theta\times\Theta)\otimes\cB(\RRR^{d})$ of the Borel $\sigma$-algebra $\cB(\Theta\times\Theta)$ and the Borel $\sigma$-algebra $\cB(\RRR^{d})$ on $\RRR^{d}$ due to (A1). Furthermore $|G(\theta,\cdot) - G(\vartheta,\cdot)|^{2}$ is $\MP^{Z}$-integrable for $\theta, \vartheta$. Then by Tonelli's theorem, $\phi(\theta,\vartheta) := \|G(\theta,\cdot) - G(\vartheta,\cdot)\|_{\MP^{Z},2}$ defines a 
mapping $\phi:\Theta\times\Theta\rightarrow\RRR$ which is Borel measurable, and thus $\Gamma_{\delta}$ is a Borel subset of $\Theta\times\Theta$. So in view of Lemma \ref{Fk measurable class} the set $\FFF^{\Theta}_{\Gamma_{\delta}}$ is a $\MP^{Z}$-measurable class. We also know from Lemma \ref{Fk measurable class} that $\FFF^{\Theta,2}_{\Theta\times\Theta}$ is a $\MP^{Z}$-measurable class. 
\medskip

Since in addition $\overline{J}(\FFF^{\Theta},\xi,1)$ is finite, we are in the position to apply Theorem 2.5.2 from \cite{vanderVaartWellner1996}. According to this result 
we may conclude immediately the entire statement 2) which completes the proof.
\hfill$\Box$
\bigskip

Let us turn over to the proof of Theorem \ref{asymptoticDistributionRiskNeutral}. It is a straightforward consequence of Theorem \ref{main technical result} via the functional delta method.
\medskip

\noindent
\textbf{Proof of Theorem \ref{asymptoticDistributionRiskNeutral}:}\\[0.1cm]
In view of the convergence result for $(\sqrt{n}~Y_{n})_{n\in\NNN}$ by Theorem \ref{main technical result} the functional delta method for infimal value mappings on $l^{\infty}(\Theta)$ (see \cite[Proposition 1 with Theorem 1]{Roemisch2006}) yields the claimed convergence of the sequence \eqref{sequence} in Theorem \ref{asymptoticDistributionRiskNeutral}. The remainig part of Theorem \ref{asymptoticDistributionRiskNeutral} is an obvious consequence.  
\hfill$\Box$
\subsection{Proof of Proposition \ref{missing link I}}
\label{Beweis missing link I}
Let us remind the random processes $X_{n}$ and the mapping $\psi$ introduced at the beginning of this section. For abbreviation we introduce the sequence $(V_{n})_{n\in\NNN}$ of random processes $V_{n}: \Theta\times\Omega\rightarrow\RRR$, defined by
$$
V_{n}(\theta,\omega) =  \frac{1}{n} 
\sum_{j=1}^{n}\Big(G\big(\theta,Z_{j}(\omega)\big) - X_{n}(\omega,\theta)\Big)^{+} - 
\EEE\big[\big(G\big(\theta,Z_{1}\big) - X_{n}(\omega,\theta)\big)^{+}\big].
$$
These random processes are building blocks of the following mappings
$$
W_{n}: \Theta\times\Omega\rightarrow\RRR,~(\theta,\omega)\mapsto \sqrt{n}~\big[V_{n}(\theta,\omega)  
- \frac{1}{n}\sum_{j=1}^{n}\big(G_{1}(\theta,Z_{j}(\omega)) - \EEE\big[G_{1}(\theta,Z_{1})\big]\big)\big]~(n\in\NNN).
$$
The sequence $(U_{n})_{n\in\NNN}$ of mappings $U_{n}: \Theta\times\Omega\rightarrow\RRR$, defined by$$
U_{n}(\theta,\omega) = \sqrt{n}~\EEE\big[\big(G(\theta,Z_{1})- X_{n}(\omega,\theta)\big)^{+} - G_{1}(\theta,Z_{1})\big] + \sqrt{n}~\overline{F}_{\theta}\big(\psi(\theta)\big)~[X_{n}(\omega,\theta) - \psi(\theta)]
$$
will also play an important role.
\smallskip

We start with the following observation 
\begin{align*}
\sqrt{n}~\Big[\cR_{\rho_{1,a}}\big(\hat{F}_{n,\theta}\big)_{|\omega} - \frac{1}{n}\sum_{j=1}^{n}\widehat{G}_{1,a}\big(\theta,Z_{j}(\omega)\big) \Big]
= 
a ~W_{n}(\theta,\omega) + a ~U_{n}(\theta,\omega) 
\end{align*}
for $\omega\in\Omega$ and $\theta\in\Theta$. 
In particular
\begin{align}
&\nonumber
\sqrt{n}~\Big|\inf_{\theta\in\Theta}\cR_{\rho_{1,a}}\big(\hat{F}_{n,\theta}\big)_{|\omega} - \inf_{\theta\in\Theta}\frac{1}{n}\sum_{j=1}^{n}\widehat{G}_{1,a}\big(\theta,Z_{j}(\omega)\big) \Big|\\
&\label{Ausgangspunkt} 
\leq 
a~\sup_{\theta\in\Theta}|W_{n}(\theta,\omega)| + a~\sup_{\theta\in\Theta}|U_{n}(\theta,\omega)|
\quad\mbox{for}~\omega\in\Omega.
\end{align}
For further preparation, combining Theorem \ref{main technical result} with Egorov's theorem, we may select some sequence $(\Omega_{k})_{k\in\NNN}$ in $\cF$ satisfying
\begin{align}
\label{Egorov}
\MP(\Omega_{k}) \geq 1 - 1/2^{k}\quad\mbox{and}\quad\lim_{n\to\infty}\sup_{\omega\in\Omega_{k}}\sup_{\theta\in\Theta}\big|X_{n}(\omega,\theta) - \psi(\theta)\big| = 0\quad\mbox{for}~ k\in\NNN.
\end{align}
The next result deals with the asymptotics of the first summand in \eqref{Ausgangspunkt}.
\begin{proposition}
\label{eliminating estimations}
Let (A1), (A2) be fulfilled, where the mapping $\xi$ from (A2) is square $\MP^{Z}$-integrable. Furthermore let $F_{\theta}$ be continuous at $\EEE[G(\theta,Z_{1})]$ for every $\theta\in\Theta$.   
Using notation \eqref{Entropie-Integral I}, if $\overline{J}(\FFF^{\Theta},\xi,1)$ is finite, then 
$$
\lim_{n\to\infty}\MP^{*}\big(\big\{\sup_{\theta\in\Theta}|W_{n}(\theta,\cdot)| > \varepsilon\big\}\big) = 0\quad\mbox{for}~\varepsilon > 0,
$$
where $\MP^{*}$ denotes the outer probability of $\MP$.
\end{proposition}
\begin{proof}
%
Let for $\varepsilon > 0$ and $n\in\NNN$ define the set $B_{n}^{\varepsilon}$ to consist of all $\omega\in\Omega$ such that $\sup_{\theta\in\Theta}|W_{n}(\theta,\omega)| > \varepsilon$. We select a sequence $(\Omega_{k})_{k\in\NNN}$ of events as in \eqref{Egorov}. 
By choice of these events it suffices to show $\MP^{*}(B_{n}^{\varepsilon}\cap\Omega_{k})\to 0$ for every $k\in\NNN$. So let us fix any $k\in\NNN$.
\smallskip

Set $\mathcal{I} := \big[-1 - \EEE[\xi(Z_{1})], 1 + \EEE[\xi(Z_{1})]\big]$ and let us remind the mapping $\overline{G}_{\mathcal{I}}$ and random processes $\overline{X}_{n}^{\mathcal{I}}$ defined in Corollary \ref{technical corollary}. Note that 
$V_{n}(\theta,\omega) = \overline{X}_{n}^{\mathcal{I}}\big((\theta,X_{n}(\omega,\theta)),\omega\big)$ holds for $X_{n}(\omega,\theta)\in\mathcal{I}$.
\smallskip

From Corollary \ref{technical corollary} we already know that the sequence $(\sqrt{n}~\overline{X}^{\mathcal{I}}_{n})_{n\in\NNN}$ converges in law to some tight Gaussian random element of $l^{\infty}(\Theta\times\mathcal{I})$. This means 
\begin{equation}
\label{equicontinuity 1}
\lim_{\delta\searrow 0}\limsup_{n\to\infty}\MP^{*}\Big(\Big\{\sup_{(\theta,t),(\vartheta,s)\in \Theta\times\mathcal{I}\atop \overline{d}((\theta,t),(\vartheta,s)) < \delta}\big|\sqrt{n}~\overline{X}^{\mathcal{I}}_{n}\big((\theta,t),\cdot\big)- \sqrt{n}~\overline{X}^{\mathcal{I}}_{n}\big((\vartheta,s),\cdot\big)\big| > \varepsilon\Big\}\Big) = 0,
\end{equation}
where $\overline{d}$ denotes the semimetric on $\Theta\times\mathcal{I}$, defined by
$$
\overline{d}\big((\theta,t),(\vartheta,s)\big) = \sqrt{\vari\Big(\overline{G}_{\mathcal{I}}\big((\theta,t),Z_{1}\big) - \overline{G}_{\mathcal{I}}\big((\vartheta,s),Z_{1}\big)\Big)}
$$
(see e.g. \cite[Example 1.5.10]{vanderVaartWellner1996}).
\smallskip

In view of \eqref{Egorov} there is some $n_{0}\in\NNN$ such that $X_{n}(\omega,\theta)\in\mathcal{I}$ for $\omega\in\Omega_{k}$, $\theta\in\Theta$, and $n\in\NNN$ with $n\geq n_{0}$. This implies
\begin{equation}
\label{equicontinuity 2}
|W_{n}(\theta,\omega)| = \Big|\sqrt{n}~\overline{X}^{\mathcal{I}}_{n}\Big(\big(\theta,X_{n}(\omega,\theta)\big),\omega\Big)- \sqrt{n}~\overline{X}^{\mathcal{I}}_{n}\Big(\big(\vartheta,\psi(\theta)\big),\omega\Big)\Big|
\end{equation}
for $\omega\in\Omega_{k}$, $\theta\in\Theta$, and $n\in\NNN$ with $n\geq n_{0}$. Furthermore by \eqref{Egorov}
\begin{equation}
\label{equicontinuity 3}
\limsup_{n\to\infty}\sup_{\theta\in\Theta}\sup_{\omega\in\Omega_{k}}\overline{d}\Big(\big(\theta,X_{n}(\omega,\theta)\big),\big(\theta,\psi(\theta)\big)\Big)\leq 
\lim_{n\to\infty}\sup_{\theta\in\Theta}\sup_{\omega\in\Omega_{k}}\big|X_{n}(\omega,\theta) - \psi(\theta)\big| = 0.
\end{equation}
Combining \eqref{equicontinuity 1}, \eqref{equicontinuity 2} and \eqref{equicontinuity 3}, we end up with $\MP^{*}(B_{n}^{\varepsilon}\cap\Omega_{k})\to 0$ for $n\to\infty$ which completes the proof.
\end{proof}
Let us turn over to the asymptotics of the second summand in \eqref{Ausgangspunkt}. 
\begin{proposition}
\label{second summand}
Let (A1) - (A3) be fulfilled, where the mapping $\xi$ from (A2) is square $\MP^{Z}$-integrable. Furthermore let $F_{\theta}$ be continuous at $\EEE[G(\theta,Z_{1})]$ for every $\theta\in\Theta$.   
Using notation \eqref{Entropie-Integral I}, if $\overline{J}(\FFF^{\Theta},\xi,1)$ is finite, then 
$$
\lim_{n\to\infty}\MP^{*}\big(\big\{\sup_{\theta\in\Theta}|U_{n}(\theta,\omega)| > \varepsilon\big\}\big) = 0\quad\mbox{for}~\varepsilon > 0,
$$
where $\MP^{*}$ stands for the outer probability of $\MP$.
\end{proposition}
\begin{proof}
Let us remind the sequence $(Y_{n})_{n\in\NNN}$ introduced at the beginning of Section \ref{basic result}, and recall that $\|Y_{n}\|_{\Theta,\infty}$ is a random variable on $\OFP$. We may apply Theorem \ref{main technical result} to conclude that the sequence $(\sqrt{n} Y_{n})_{n\in\NNN}$ converges in law to some tight random element of $l^{\infty}(\Theta)$. Since the norm $\|\cdot\|_{\Theta,\infty}$ is a continuous function on $l^{\infty}(\Theta)$, the application of the continuous mapping theorem for convergence in law (see \cite[Theorem 1.11.1]{vanderVaartWellner1996}) yields that $(\sqrt{n}~ \|Y_{n}\|_{\Theta,\infty})_{n\in\NNN}$ converges weakly to some random variable. In particular, by Prokhorov's theorem, this sequence of random variables is uniformly tight. Hence we may find some strictly increasing sequence $(a_{k})_{k\in\NNN}$ of positive real numbers such that the inequality $\MP(\{\sqrt{n}~\|Y_{n}\|_{\Theta,\infty}\leq a_{k}\})\geq 1 - 1/2^{k}$ holds for $k,n\in\NNN$.
\smallskip

For $k,n\in\NNN$, $\theta\in\Theta$ and $\omega\in A_{kn} := \{\sqrt{n}~\|Y_{n}\|_{\Theta,\infty}\leq a_{k}\}$ we have
\begin{align}
&\nonumber
|U_{n}(\theta,\omega)|\\
&\nonumber
= 
\sqrt{n}~\big|\EEE\big[\big(G(\theta,Z_{1}) - \psi(\theta) - Y_{n}(\theta,\omega)\big)^{+}- \big(G(\theta,Z_{1}) - \psi(\theta)\big)^{+}\big] + \overline{F}_{\theta}\big(\psi(\theta)\big) Y_{n}(\theta,\omega)\big|\\
&\nonumber
= 
\sqrt{n}~\int_{- Y_{n}(\theta,\omega)^{-}}^{0}\hspace*{-1cm}\big[\overline{F}_{\theta}\big(\psi(\theta) + t\big) - \overline{F}_{\theta}\big(\psi(\theta)\big)\big]~dt 
~-~ 
\sqrt{n}~\int^{Y_{n}(\theta,\omega)^{+}}_{0}\hspace*{-1cm}\big[\overline{F}_{\theta}\big(\psi(\theta) + t\big) - \overline{F}_{\theta}\big(\psi(\theta)\big)\big]~dt\\ 
&\nonumber
= 
\sqrt{n}~\int^{Y_{n}(\theta,\omega)^{-}}_{0}\hspace*{-1cm}\big[\overline{F}_{\theta}\big(\psi(\theta) - t\big) - \overline{F}_{\theta}\big(\psi(\theta)\big)\big]~dt 
~-~ 
\sqrt{n}~\int^{Y_{n}(\theta,\omega)^{+}}_{0}\hspace*{-1cm}\big[\overline{F}_{\theta}\big(\psi(\theta) + t\big) - \overline{F}_{\theta}\big(\psi(\theta)\big)\big]~dt\\ 
&\nonumber 
\leq 
\sqrt{n}~\int^{a_{k}/\sqrt{n}}_{0}\big[F_{\theta}\big(\psi(\theta) + t\big) - F_{\theta}\big(\psi(\theta) - t\big)\big]~dt\\ 
&\label{Integraldarstellung}
=
\int_{0}^{a_{k}}\big[F_{\theta}\big(\psi(\theta) + s/\sqrt{n}\big) - F_{\theta}\big(\psi(\theta) - s/\sqrt{n}\big)\big]~ds =: b_{k,n}(\theta),
\end{align} 
where in the last step we have used the change of variable formula.
\smallskip

Next, let $(\theta_{k,n})_{n\in\NNN}$ be a sequence in $\Theta$ satisfying $b_{k,n}(\theta_{k,n}) > \sup_{\vartheta\in\Theta}b_{k,n}(\vartheta) - 1/n$ for every $n\in\NNN$. By compactness of $\Theta$ any subsequence $(\theta_{k,i(n)})_{n\in\NNN}$ of $(\theta_{k,n})_{n\in\NNN}$ has a further subsequence $(\theta_{k,\varphi(i(n))})_{n\in\NNN}$ which converges to some $\theta\in\Theta$ w.r.t. the Euclidean metric. The mapping $\psi$ is continuous under assumptions (A1) - (A3) due to Lemma \ref{basic observations}. Hence $\psi(\theta_{k,\varphi(i(n))}) + s/\sqrt{\varphi(i(n))}\to \psi(\theta)$ and $\psi(\theta_{k,\varphi(i(n))})  - s/\sqrt{\varphi(i(n))}\to \psi(\theta)$ for $s\in [0,a_{k}]$. 
Since $F_{\theta}$ is assumed to be continuous at $\psi(\theta)$, we may invoke Lemma \ref{continuity convergence} to end up with 
$\big[F_{\theta_{k,\varphi(i(n))}}\big(\psi(\theta_{k,\varphi(i(n))}) + s/\sqrt{\varphi(i(n))}\big)- F_{\theta_{k,\varphi(i(n))}}\big(\psi(\theta_{k,\varphi(i(n))}) - s/\sqrt{\varphi(i(n))}\big)\big]\to 0$ for $s\in [0,a_{k}]$. Then by dominated convergence theorem $b_{k,\varphi(i(n))}(\theta_{\varphi(i(n))})\to 0$, and thus $\sup_{\vartheta\in\Theta}b_{k,\varphi(i(n))}(\vartheta)\to 0$ due to the choice of the 
sequence $(\theta_{n})_{n\in\NNN}$. Therefore we may conclude
\begin{equation}
\label{Konvergenz Maxima}
\lim_{n\to\infty}\sup_{\vartheta\in\Theta}b_{k,n}(\vartheta) = 0\quad\mbox{for}~k\in\NNN.
\end{equation}
Putting \eqref{Integraldarstellung} and \eqref{Konvergenz Maxima} together we may conclude
\begin{align*}
\limsup_{n\to\infty}\MP^{*}\big(\big\{\sup_{\theta\in\Theta}|U_{n}(\theta,\cdot)| > \varepsilon\big\}\big)\leq 
\limsup_{n\to\infty}\MP^{*}\big(\big\{\sup_{\theta\in\Theta}|U_{n}(\theta,\cdot)| > \varepsilon\big\}\cap A_{kn}\big) + 1/2^{k} = 1/2^{k}
\end{align*}
for $\varepsilon > 0$ and $k\in\NNN$. This completes the proof.
\end{proof}
Now we are ready to show Proposition \ref{missing link I}.
\medskip

\noindent
\textbf{Proof of Proposition \ref{missing link I}:}\\[0.1cm]
Putting \eqref{Ausgangspunkt} together with Proposition \ref{eliminating estimations} and Proposition \ref{second summand}, we obtain for any $\varepsilon > 0$
\begin{align*}
&
\limsup_{n\to\infty}\MP^{*}\Big(\Big\{\big|\sqrt{n}\inf_{\theta\in\Theta}\cR_{\rho_{1,a}}\big(\hat{F}_{n,\theta}\big) - \sqrt{n}\inf_{\theta\in\Theta}\frac{1}{n}\sum_{j=1}^{n}\widehat{G}_{1,a}(\theta,Z_{j})\big| > \varepsilon\Big\}\Big)\\
&\leq
\limsup_{n\to\infty}\MP^{*}\big(\big\{\sup_{\theta\in\Theta}|W_{n}(\theta,\cdot)| + \sup_{\theta\in\Theta}|U_{n}(\theta,\cdot)| > \varepsilon/a\big\}\big)\\ 
&\leq 
\limsup_{n\to\infty}\big[\MP^{*}\big(\big\{\sup_{\theta\in\Theta}|W_{n}(\theta,\cdot)|  > \varepsilon/(2 a)\big\}\big)  + \MP^{*}\big(\big\{ \sup_{\theta\in\Theta}|U_{n}(\theta,\cdot)| > \varepsilon/(2 a)\big\}\big)\big] = 0.
\end{align*}
This completes the proof. 
\hfill$\Box$

\begin{appendix}
\section{Appendix}
\label{certainty equivalents}
Throughout this section we shall us assumptions and notations from Section \ref{generalized divergence risk measures}.
\smallskip

Let us introduce the \textit{function of certainty equivalents} defined by
$$
\varphi_{\Phi^{*}}: \FFF_{\Phi^{*}}\times\RRR\rightarrow\RRR,~(F,x)\mapsto \int_{0}^{1}\eins_{]0,1[}(u)~ \Phi^{*}\big(\Flinks(u) + x\big)~du - x
$$
where $\Flinks$ denotes the left-continuous quantile function of $F$. Note that this mapping is well-defined because $F\in\FFF_{\Phi^{*}}$ if and only if $\Flinks(U)\in H^{\Phi^{*}}$ for any random variable $U$ on the atomless probability space $\OFP$ which is uniformly distributed on $]0,1[$. Moreover, it satisfies the following basic properties.
\begin{proposition}
\label{function of certainty equivalents basic properties}
For $F\in \FFF_{\Phi^{*}}$ the mapping $\varphi_{\Phi^{*}}(F,\cdot)$ is convex, and 
the set 
of minimizers of $\varphi_{\Phi^{*}}(F,\cdot)$ is a nonvoid compact interval in $\R$. 
\end{proposition}
\begin{proof}
Let us fix any $F\in\FFF_{\Phi^{*}}$. Convexity of $\varphi_{\Phi^{*}}(F,\cdot)$ is obvious due to convexity of $\Phi^{*}$. Moreover by the assumptions on $\Phi$ we already know that $\Phi^{*}$ fulfills
$$
\lim_{x\to -\infty}\big[\Phi^{*}(x) - x\big] = \lim_{x\to \infty}\big[\Phi^{*}(x) - x\big] = \infty
$$
(see \cite[Lemma A.1]{BelomestnyKraetschmer2016}). Furthermore $\Flinks$ is integrable w.r.t. the uniform distribution on $]0,1[$ because $F\in\F_{\Phi^{*}}$. Then the application of Jensen's inequality yields
\begin{eqnarray*}
&&
\lim_{x\to -\infty}\varphi_{\Phi^{*}}(x,F)
\geq\lim_{x\to -\infty} \left[\Phi^{*}\left(\int_{0}^{1}\eins_{]0,1[}(u)\Flinks(u)~du + x\right) - x\right]
= \infty.\\
&& 
\lim_{x\to \infty} \varphi_{\Phi^{*}}(x,F) 
\geq\lim_{x\to \infty} \left[\Phi^{*}\left(\int_{0}^{1}\eins_{]0,1[}(u)\Flinks(u)~du + x\right) - x\right]
= \infty.
\end{eqnarray*}
This shows that any level set $\varphi_{\Phi^{*}}(F,\cdot)^{-1}(]-\infty,\gamma])$ is a bounded subset of $\R$. Moreover, these sets are also closed because the mapping $\varphi_{\Phi^{*}}(F,\cdot)$ is convex and thus continuous. Hence all these level sets are compact subsets of $\R$, and in particular the set of minimizers 
is a nonvoid compact subset of $\R$ due to continuity of $\varphi_{\Phi^{*}}(F,\cdot)$. Finally, the set of minimizers 
is also convex since $\varphi_{\Phi^{*}}(F,\cdot)$ is a convex mapping. This completes the proof.
\end{proof}
The next result provides a criterion, dependent on $\Phi$ and $\Phi^{*}$ only, to ensure that a mapping $\varphi_{\Phi^{*}}(F,\cdot)$ has a unique minimizer.
\begin{proposition}
\label{unique minimizer}
Let $\Phi(0) = 0$, and let $\Phi^{*}$ be strictly convex on $]0,\infty[$. Then the mapping $\varphi_{\Phi^{*}}(F,\cdot)$ has a unique minimizer for every $F\in\F_{\Phi^{*}}$.
\end{proposition}
\begin{proof}
Let us fix any $F\in\F_{\Phi^{*}}$. We have $\Phi^{*}(y) = 0$ for every $y \leq 0$ because $\Phi(0) = 0$. Hence $\Phi^{*}(y)\geq 0$ for $x\in\RRR$, and $\Phi^{*}\big(\Flinks(u) + x\big) = 0$ if $u < F(-x)$. In particular
\begin{equation}
\label{reduzierte Darstellung}
\varphi_{\Phi^{*}}(F,x) = \int_{0}^{1}\eins_{]F(-x),1[}(u)~\Phi^{*}\big(\Flinks(u) +x\big)~du - x\quad\mbox{for}~x\in\RRR.
\end{equation}
By way of contradiction let us assume that there are at least two minimizers of $\varphi_{\Phi^{*}}(F,\cdot)$. In view of Proposition \ref{function of certainty equivalents basic properties} the set of minimizers is an interval. Since by \eqref{reduzierte Darstellung} we have $\varphi_{\Phi^{*}}(F,x) = - x$ if $F(-x) = 1$, we may find minimizers $x_{0}, x_{1}$ with $x_{0} < x_{1}$ such that 
$F(-x_{1})\leq F(-x_{0}) < 1$. Furthermore $x_{2} := (x_{0} + x_{1})/2$ is a minimizer too, implying $\varphi_{\Phi^{*}}(F,x_{2}) = [\varphi_{\Phi^{*}}(F,x_{0}) + \varphi_{\Phi^{*}}(F,x_{1})]/2$. Since $\Phi^{*}$ is nonnegative and convex, we may conclude from \eqref{reduzierte Darstellung}
\begin{align*}
0 
&\geq 
\int_{0}^{1}\eins_{]F(-x_{2}),1[}(u)~\big[\Phi^{*}\big(\Flinks(u) +x_{2}\big) - \sum_{i=0}^{1}\Phi^{*}\big(\Flinks(u) +x_{i}\big)/2\big]~du\\
&= 
\frac{1}{2}~\int_{0}^{1}\hspace*{-0.2cm}\big[\eins_{]F(-x_{1}),F(-x_{2})[}(u)~\Phi^{*}\big(\Flinks(u) +x_{1}\big)
- 
\eins_{]F(-x_{2}),F(-x_{0})[}(u)~\Phi^{*}\big(\Flinks(u) +x_{0}\big)\big] ~du\\
&= 
\frac{1}{2}~\int_{0}^{1}\eins_{]F(-x_{1}),F(-x_{2})[}(u)~\Phi^{*}\big(\Flinks(u) +x_{1}\big)
~du \geq 0.
\end{align*}
Then by convexity of $\Phi^{*}$ again,  
$
\sum_{i=0}^{1}\Phi^{*}\big(\Flinks(u) +x_{i}\big)/2 = \Phi^{*}\big(\Flinks(u) +x_{2}\big)
$
for almost all $u\in ]F(-x_{0}),1[$. This contradicts strict convexity of $\Phi^{*}$ on $]0,\infty[$ because $\Flinks(u) +x_{i} > 0$ for $i\in\{0,1,2\}$ and $u\in ]F(-x_{0}),1[$. Therefore, $\varphi_{\Phi^{*}}(F,\cdot)$ has at most one minimizer. This completes the proof because the set of minimizers is nonvoid due to Proposition \ref{function of certainty equivalents basic properties}. 
\end{proof}
\section{Appendix}
\label{AppendixB}
Let $\frak{G} = (\frak{G}_{t})_{t\in\TTT}$ be some centered Gaussian process on some probability space $\tOFP$ with intrinsic semimetric $d_{\TTT}$ on $\TTT$, defined by 
$d_{\TTT}(t,\overline{t}) = \sqrt{\vari(\frak{G}_{t} - \frak{G}_{\overline{t}})}$. We assume $\TTT = \bigcup_{k\in\NNN}\TTT_{k}$ for some isotone sequence $(\TTT_{k})_{k\in\NNN}$ of $d_{\TTT}$-totally bounded subsets of $\TTT$. Then we find an isotone sequence $(\overline{\TTT}_{k})_{k\in\NNN}$ of at most countable sets such that $\overline{\TTT}_{k}\subseteq\TTT_{k}$ is dense w.r.t. $d_{\TTT}$ for any $k\in\NNN$. Set $\overline{\TTT} := \bigcup_{k\in\NNN}\overline{\TTT}_{k}$, and note that this set is $d_{\TTT}$-dense.
\smallskip

The aim of this section is to find a convenient criterion which guarantees that $\frak{G}$ has version which is $d_{\TTT}$-uniformly continuous on every $\TTT_{k}$. Our result is based on the following auxiliary consideration.
\begin{lemma}
\label{extension result}
If a mapping $f:\overline{\TTT}\rightarrow\RRR$ is $d_{\TTT}$-uniformly continuous on $\overline{\TTT}_{k}$ for every $k\in\NNN$, then there exists a unique extension $\hat{f}:\TTT\rightarrow\RRR$ of $f$ which is uniformly continuous w.r.t. $d_{\TTT}$ on $\TTT_{k}$ for every $k\in\NNN$.
\end{lemma}
\begin{proof}
It is known that every uniformly continuous mapping on a semimetric space may be extended uniquely to a uniformly continuous mapping on the completion of the semimetric space 
(see \cite[Theorem 11.3.4 with Theorem 9.2.2]{Wilansky1970}). Since $\overline{\TTT}_{k}$ is a $d_{\TTT}$-dense subset of $\TTT_{k}$ both spaces have the same $d_{\TTT}$-completion for $k\in\NNN$. Hence we may have for any $k\in\NNN$ a unique $d_{\TTT}$-uniformly continuous extension $f_{k}:\TTT_{k}\rightarrow\RRR$ of the restriction $f_{|\overline{\TTT}_{k}}$. Since $\TTT_{k}\subseteq\TTT_{k+1}$ we may observe that $f_{k}$ and $f_{k+1}$ coincide on $\TTT_{k}$ for $k\in\NNN$. Then we may define the mapping $\hat{f}:\TTT\rightarrow\RRR$ by $\hat{f}(t) = f_{k}(t)$ if $t\in\TTT_{k}\setminus\TTT_{k-1}$ for some $k\in\NNN$, where $\TTT_{0} := \emptyset$. This mapping is as required.
\end{proof}
We may now formulate the announced criterion on continuous modifications of $\frak{G}$.
\begin{proposition}
\label{continuous modification}
Let $\cU^{k}_{\delta}$ be the set of all $(t,\overline{t})\in\overline{\TTT}_{k}\times\overline{\TTT}_{k}$ such that $d_{\TTT}(t,\overline{t}) < \delta$ for $k\in\NNN$ and $\delta > 0$. If $\lim\limits_{\delta\searrow 0}\EEE\big[\sup_{(t,\overline{t})\in\cU^{k}_{\delta}}|\frak{G}_{t} - \frak{G}_{\overline{t}}|\big] = 0$ for $k\in\NNN$, then there exists a version $\widehat{\frak{G}} = (\widehat{\frak{G}}_{t})_{t\in\TTT}$ of $\mathfrak{G}$ whose paths are uniformly continuous w.r.t. $d_{\TTT}$ on $\TTT_{k}$ for any $k\in\NNN$.
\end{proposition}
\begin{proof}
We are in the position to draw on Theorem 4 from \cite{KuelbsEtAl2013} to select some sequence $(\overline{\Omega}_{k})_{k\in\NNN}$ in $\overline{\cF}$ such that $\overline{\MP}(\overline{\Omega}_{k}) = 1$, and $\mathfrak{G}_{\cdot}(\overline{\omega})$ is uniformly continuous w.r.t. $d_{\TTT}$ on $\overline{\TTT}_{k}$ for $\overline{\omega}\in\overline{\Omega}_{k}$. Define the event $\widehat{\Omega} := \bigcap_{k=1}^{\infty}\overline{\Omega}_{k}\in\overline{\cF}$ which satisfies $\overline{\MP}(\widehat{\Omega}) = 1$. Then by Lemma \ref{extension result} there exists for any $\overline{\omega}\in\widehat{\Omega}$ a unique mapping $\widehat{\frak{G}}_{\cdot}(\overline{\omega}):\TTT\rightarrow\RRR$ which is $d_{\TTT}$-uniformly continuous on $\TTT_{k}$ for $k\in\NNN$ such that $\widehat{\frak{G}}_{t}(\overline{\omega}) = \frak{G}_{t}(\overline{\omega})$ for $t\in\overline{\TTT}$. Then, setting $\widehat{\frak{G}}_{t}(\overline{\omega}) := 0$ for $\overline{\omega}\in\overline{\Omega}\setminus\widehat{\Omega}$, we have constructed a stochastic process $\widehat{\frak{G}} = (\widehat{\frak{G}}_{t})_{t\in\TTT}$ on $\tOFP$ which has $d_{\TTT}$-continuous paths on $\TTT_{k}$ for $k\in\NNN$. It is remains to show that it is a version of $\frak{G}$. 
\smallskip

Let $t\in\TTT$. If $t\in\overline{\TTT}$, then by construction $\frak{G}_{t}$ and $\widehat{\frak{G}}_{t}$ coincide on $\widehat{\Omega}$ so that they differ on a $\overline{\MP}$-null set only. So let $t\in\TTT\setminus\overline{\TTT}$. Then $t\in\TTT_{k}$ for some $k\in\NNN$ and $d_{\TTT}(t,t_{l})\to 0$ for a sequence $(t_{l})_{l\in\NNN}$ in $\overline{\TTT}_{k}$. Since $d_{\TTT}(t,t_{l})^{2} = \EEE[(\frak{G}_{t} - \frak{G}_{t_{l}})^{2}]$ for $l\in\NNN$ we may conclude from Vitalis' theorem (see \cite[Proposition 21.4]{Bauer2001}) that $\frak{G}_{t_{l}}\to\frak{G}_{t}$ in probability. This implies that $\widehat{\frak{G}}_{t_{l}}\to\frak{G}_{t}$ in probability because $\widehat{\frak{G}}_{t_{l}} = \frak{G}_{t_{l}}$ $\overline{\MP}$-a.s. for $l\in\NNN$. Moreover, $\widehat{\frak{G}}_{t_{l}}\to\widehat{\frak{G}}_{t}$ in probability by path-continuity of $\widehat{\frak{G}}$. Hence $\overline{\MP}(\{|\widehat{\frak{G}}_{t} - \frak{G}_{t}| > \varepsilon\}) = 0$ for any $\varepsilon > 0$, and thus $\overline{\MP}(\{|\widehat{\frak{G}}_{t} - \frak{G}_{t}| > 0\}) = 0$. This completes the proof.
\end{proof}
\end{appendix}
\section*{Acknowledgments}
The author thanks two anonymous referees for useful comments and suggestions which have helped to improve an earlier draft.

\section*{Declarations}

\textbf{Conflicts of interest:} The author has no competing interests to declare that are relevant to the content of this article.

\medskip\noindent
\textbf{Data Availability Statement:} No datasets were generated or analysed during the study

\medskip\noindent
\textbf{Funding:} The authors did not receive support from any organization for the submitted work.


\end{document}